\documentclass[12pt, reqno]{amsart}
\usepackage[dvips]{graphicx}
\usepackage{color}
\usepackage{amscd}

\usepackage{amsfonts,amssymb,verbatim,amsmath,amsthm,latexsym,textcomp,amscd}

\usepackage{color}
\usepackage{url}

\def\CC{\mathbb C}
\def\Chat{\hat {\mathbb C}}
 
\def\HH{\mathbb H}

\def\QQ{\mathbb Q}
\def\RR{\mathbb R}
\def\TT{\mathbb T}
\def\ZZ{\mathbb Z}

\def\a{\alpha}
\def\b{\beta}
\def\c{\gamma}
\def\d{\delta}
\def\g{\gamma}

\def\s{\sigma}

\def\th{\theta}
\def\w{\omega}
\def\z{\zeta}

\def\cal{\mathcal}

\def\B{{ \mathcal B}}
\def\C{{ \mathcal C}}
\def\D{ \mathcal D}

\def\F{{  \mathcal F}}

\def\H{{\cal H}}

\def\O{{\kappa}}

\def\S{{\mathcal S}}

\def\SCH {{\mathcal{SCH}}}
\def \Diag  {{\Delta}}

\def\P{{\cal P}}

\def\T{{\cal T}}
\def\U{{\mathcal U}}

  \def\arg{\mathop{\rm{Arg}}} 
   \def\Ax{\mathop{\rm{Ax}}}

    \def\Int{\mathop{{Int}}}

  \def\SL{SL(2,\mathbb C)}
  \def\PSL{PSL(2,\mathbb C)}
   
 \def\tr{\mathop{\rm{Tr}}}  
  \def\Tr{\mathop{\rm{Tr}}}

\def\dd{\partial}

\def\square{\hfill${\vcenter{\vbox{\hrule height.4pt \hbox{\vrule width.4pt
height7pt \kern7pt \vrule width.4pt} \hrule height.4pt}}}$}

\def\co{{\colon \thinspace}}

\newtheorem{theorem}{Theorem}[section]

\newtheorem{definition}[theorem]{Definition}
\newtheorem{lemma}[theorem]{Lemma}
\newtheorem{proposition}[theorem]{Proposition}
\newtheorem{prop}[theorem]{Proposition}
\newtheorem{corollary}[theorem]{Corollary}
\newtheorem{remark}[theorem]{Remark}

\begin{document}
%%%%%%%%%%%%%%%%%%%%%%%%%%%%%%%%%%%%%%%%%%%%%%%%%%%%%%%%%%%%%%%%%%%%%
\title{The diagonal slice of Schottky space}

\author{Caroline Series}
\address{\begin{flushleft} \rm {\texttt{C.M.Series@warwick.ac.uk \\http://www.maths.warwick.ac.uk/$\sim$masbb/} }\\ Mathematics Institute, 
 University of Warwick \\
Coventry CV4 7AL, UK \end{flushleft}}

\author{Ser Peow Tan} 
\address{\begin{flushleft} \rm {\texttt{mattansp@nus.edu.sg \\http://www.math.nus.edu.sg/$\sim$mattansp/} }\\ Department of Mathematics, National University of Singapore, 
10, Lower Kent Ridge Road
Singapore 119076
\end{flushleft}}

\author{Yasushi Yamashita} 
\address{\begin{flushleft} \rm {\texttt{yamasita@ics.nara-wu.ac.jp \\http://vivaldi.ics.nara-wu.ac.jp/$\sim$yamasita/}}\\ Department
of Information and Computer Sciences, Nara Women's University,
630-8506
Kitauoyanishi-machi, Nara-City, Japan \end{flushleft}}

\thanks{The second author is partially supported by the National University
of Singapore academic research grant R-146-000-186-112. The third author is partially supported by JSPS KAKENHI Grant Number 23540088.}   

\date{\today}

\begin{abstract} An irreducible representation of the free group on two generators $X,Y $ into $SL(2,\CC)$ is determined up to conjugation by the traces of $X,Y$ and $XY$.  If the representation is free and discrete, the resulting manifold is in general a genus-$2$ handlebody.
We study the diagonal slice of the representation variety in which 
$\tr X = \tr Y = \tr XY $. Using the symmetry, we are able to compute the Keen-Series pleating rays and thus fully determine the locus of free and discrete groups.  
We  also computationally determine the `Bowditch set' 
consisting of those parameter values for which no primitive elements in $\langle X,Y \rangle$ have traces in $[-2,2]$, and at most finitely many primitive elements have traces with absolute value at most $2$. The graphics make clear that this set is both strictly larger than, and significantly different from,  the discreteness locus.

 {\bf MSC classification: 30F40; 57M50}    

\end{abstract}

\maketitle

\section{Introduction} \label{intro}

It is well known that an  irreducible representation of the free group $F_2$ on two generators $X,Y $ into $SL(2,\CC)$ is determined up to conjugation by the traces of $X,Y$ and $XY$. More generally, if we take the GIT quotient of   all (not necessarily irreducible) representations, then the resulting $SL(2,\CC)$ character variety of $F_2$ can be identified with $\CC^3$ via these traces, see for example \cite{goldman2} and the references therein.   If the representation is free, discrete, purely loxodromic and geometrically finite, the resulting manifold is a genus-$2$ handlebody.  The collection of all such representations is known as \emph{Schottky space}, denoted  $\mathcal {SCH}$. It is a consequence of Bers' density theorem that $\mathcal {SCH}$ is the interior of the discreteness locus, see for example~\cite{canary}. It is natural to ask, for which values of   $x = \Tr X, y = \Tr Y, z= \Tr XY$  is the corresponding  representation in  $\mathcal {SCH}$?

Let $\P$ denote the set of primitive elements in $F_2$. For $(x,y,z) \in \CC^3$, let $\rho_{(x,y,z)}$ denote a choice of representation $  F_2 \to \SL$ in the conjugacy class determined by the trace triple. 
The \emph{Bowditch set} (or $BQ$-set) $\mathcal B$  is defined in ~\cite{tan_gen}  as the set of $(x,y,z) \in \CC^3$ corresponding to irreducible representations for which 
\begin{equation*} 
 \begin{split}  & \Tr \rho_{(x,y,z)}(g) \notin [-2,2]  \ \ \forall g \in \P    \ \ \mbox {\rm and}   \ \ \cr
& \{ g \in \P: |\Tr \rho_{(x,y,z)}(g)| \leq 2 \} \ \mbox {\rm is finite}.\end{split}\end{equation*}
(The exceptional case in which 
$\Tr [X,Y] = 2$ corresponds to reducible representations and is excluded from the discussion, see Remark~\ref{reducible}.)
The Bowditch set is open and ${\rm Out}(F_2)$ acts properly discontinuously on it. Clearly $\SCH \subset \mathcal B$.

Bowditch's original work~\cite{bow_mar} was on the case in which  the commutator $[X,Y] = XYX^{-1}Y^{-1}$
is parabolic and $\Tr  [X,Y] =  -2 $. He conjectured that the subsets of  $\SCH$ and $ \B$ corresponding to this restriction  coincide.
Although this has not been proven, computer pictures  indicate his conjecture may well be true.%, see~\cite{sty}.

In this paper we restrict to the special case in which $x = y = z$, which we call the \emph{diagonal slice} of the character variety, denoted $  \Diag$ and parametrized by the single complex variable $x$.  We show that in this slice, the analogue of Bowditch's conjecture is far from being true. This is illustrated in Figure~\ref{Diagonal-and-Riley-mu-3-Ray-BQ} which compares the  intersections of $\Delta$ with $\SCH$ and $\B$.
The discreteness locus is the outer region foliated by rays; these are the Keen-Series pleating rays which relate to the geometry of the convex hull boundary as explained in Section~\ref{sec:pleating} and whose closure is known to be $\overline{ \Diag \cap \SCH}$, see Theorem~\ref{thm:density}.
The Bowditch set, by contrast, is the complement of the black part.
It  is clear   that $\B \cap \Diag$ contains a large open region not in   $ \Diag \cap \SCH$, and also has different symmetries.  
In particular,   it is not hard to show that  the interval $  (2,3)$ is contained in $   \B \setminus \SCH$, see the discussion  in Section~\ref{sec:algorithm1}.

The main content of this paper is an explanation and justification of  how these plots were made, in particular to explain how we enumerated and computed the pleating rays for the symmetric genus $2$ handlebody corresponding to the trace triple $(x,x,x)$.

\begin{figure}[hbt]\includegraphics[width=6cm]{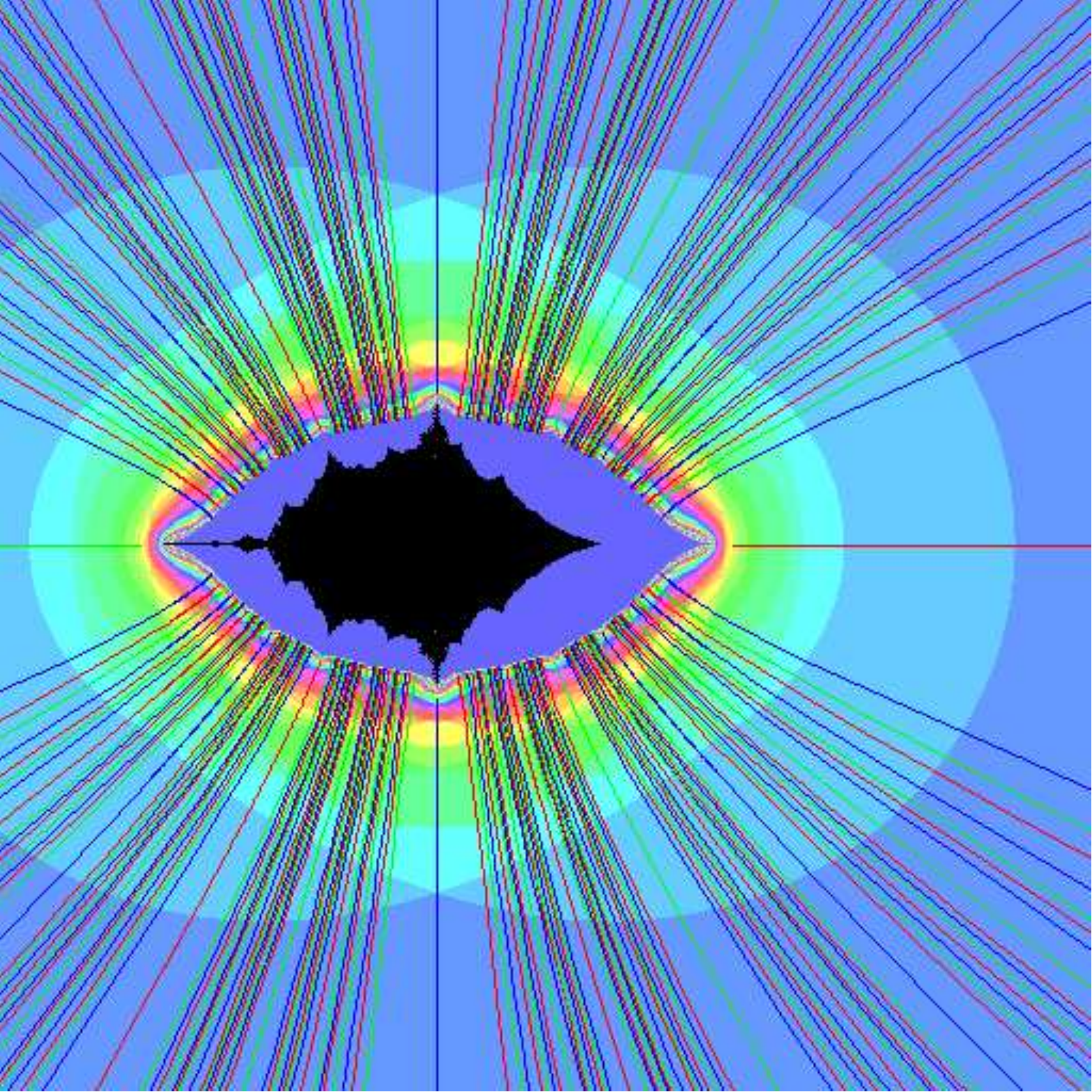}
\caption{Superposition of the discreteness locus for $\pi_1(\H)$ and the Bowditch set in the $x$-plane. The Bowditch set for the $(x,x,x)$-triple is the complement of the central black region, while the discreteness locus is the closure of the region foliated by rays. The rays  are actually computed as the pleating rays for the quotient orbifold $\S$.}\label{Diagonal-and-Riley-mu-3-Ray-BQ}
\end{figure}

To compute the Bowditch set $\B$ we use an algorithm based on the ideas in~\cite{bow_mar} and developed further  in~\cite{tan_gen}. This is explained 
in Section~\ref{sec:algorithm}.  

The discreteness problem is tackled as follows. If $(x,x,x) \in \SCH$ then the quotient $3$-manifold $\HH^3/G$ is a handlebody $\H$ with order $3$ symmetry. We use the symmetry to  reduce  the problem of finding $ \Diag \cap \SCH$ to a problem very similar to that of determining the so-called \emph{Riley slice  of Schottky space}. This is actually a space of groups on the boundary of $ \SCH$, consisting of those free, discrete and geometrically finite groups for which  the two  generators $X,Y$ are parabolic, thus contained in the slice $(2,2,z) \subset \CC^3$. The corresponding  manifold is a handlebody whose conformal boundary is a sphere with four parabolic points.   The problem of finding those $z$-values for which such a group is free, discrete and geometrically finite was  solved using the method of pleating rays in~\cite{ksriley}.  In the present case, the quotient of $\H$ by the symmetry is an orbifold $\S$ with two order $3$ cone axes, whose conformal boundary is a sphere with four order $3$ cone points. Thus similar methods enable us to find   $ \Diag \cap \SCH$ here.

Although Figure~\ref{Diagonal-and-Riley-mu-3-Ray-BQ} shows that in $\Delta$, the analogue of Bowditch's conjecture fails since $\B$  and the interior of the discreteness locus are plainly distinct, in many other  slices, see for example Figure~\ref{Figs/Riley-Ray-BQ},  the (modified) Bowditch set  and the interior of the discreteness locus appear to coincide. This  is connected to the dynamics of the action of a suitable mapping class group on representations and raises many interesting questions which we hope to address elsewhere.

The plan of the paper is as follows.  
We begin in Section~\ref{sec:markoff} with a discussion of the Markoff tree and the algorithm used to compute the Bowditch set.
In Section~\ref{sec:basicconfig} we introduce a basic geometrical construction  which conveniently 
encapsulates the $3$-fold symmetry. The quotient of the original   handlebody  $\H$ by the symmetry is a ball with two order $3$ cone axes. This orbifold $\S$ has a further $4$-fold symmetry group whose quotient is again a topological ball.  Our construction allows us to write down specific $\SL$ representations of all the groups involved with ease.
In Section~\ref{sec:discrete} we turn to the discreteness question. After reducing the problem to one on $\S$, 
we briefly review material from the Keen-Series theory of pleating rays and
recall  what is needed from~\cite{ksriley}, allowing us to
apply a similar proof in the present context.  Section~\ref{sec:torustree}, not strictly logically necessary for our development, explains how we did our trace computations in practice, by relating the problem to one on a commensurable torus with a single cone point of angle $4\pi/3$.

\section{The Markoff tree and the Bowditch set} \label{sec:markoff}
Let $A =  \begin{pmatrix} a & b \cr c & d\end{pmatrix} \in \SL$ so that $ad-bc= 1$. As usual we define its trace $\Tr A = a+d$.

Let $F_2 = \langle X,Y|  \  \rangle$ be the free group on two generators.
It is well known that a representation $\rho\co F_2 \to \SL$ is determined up to conjugation (modulo taking the GIT quotient under the conjugation action, see~\cite{goldman2})
by the three traces $ x = \tr X, y = \tr Y, z = \tr XY$. In fact, 
given $x,y,z \in \CC$   we can define a representation $\rho_{x,y,z}\co F_2  \to \SL $ by 
$\rho(X) = \begin{pmatrix} x & 1 \cr -1 & 0 \end{pmatrix}, \ \rho(Y) = \begin{pmatrix} 0 & \xi \cr -\xi^{-1} & y \end{pmatrix}$ where $ z = -(\xi + \xi^{-1})$.  Clearly with this definition, $ \Tr X = x, \Tr Y = y$ and $\Tr XY = z$.

\subsection{The Markoff Tree}
For matrices $U,V \in \SL$  set $u = \Tr U, v = \Tr V, w = \Tr UV$. Recall the trace relations:
\begin{equation}\label{eqn:inverse} \Tr UV^{-1} = uv-w \end{equation} and 
\begin{equation} \label{eqn:commreln}  u^2+v^2+w^2 = uvw + \Tr {[U,V]} +2.
\end{equation}
Setting $\mu =  \Tr {[U,V]} + 2$, this last equation takes the form 
$$u^2+v^2+w^2 - uvw = \mu.$$

Let $F_2 = \langle X,Y| \ \  \rangle$ as above. An element $U \in F_2$ is \emph{primitive} if it is a member  of a generating pair; we denote the set of all primitive elements by $\P$.  The conjugacy classes of primitive elements are enumerated by $\hat \QQ = \QQ \cup \infty$ and are conveniently organised relative to the Farey diagram $\F$ as shown in Figure~\ref{fig:farey}. 
This consists of  the images of the ideal triangle with vertices at $1/0,0/1$ and $1/1$ under the action of $SL(2,\ZZ)$ on the upper half plane, suitably conjugated to the position shown in the disk. The label $p/q$ in the disk is just the conjugated image of the actual point $p/q \in \RR$.

\begin{figure}[ht]
\includegraphics[width=5.5cm]{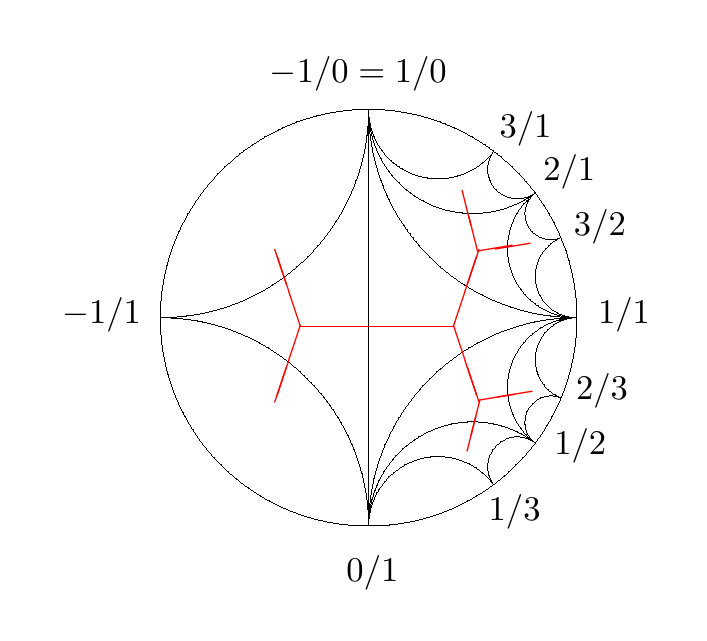}
\hspace{1cm}
\includegraphics[width=5.5cm]{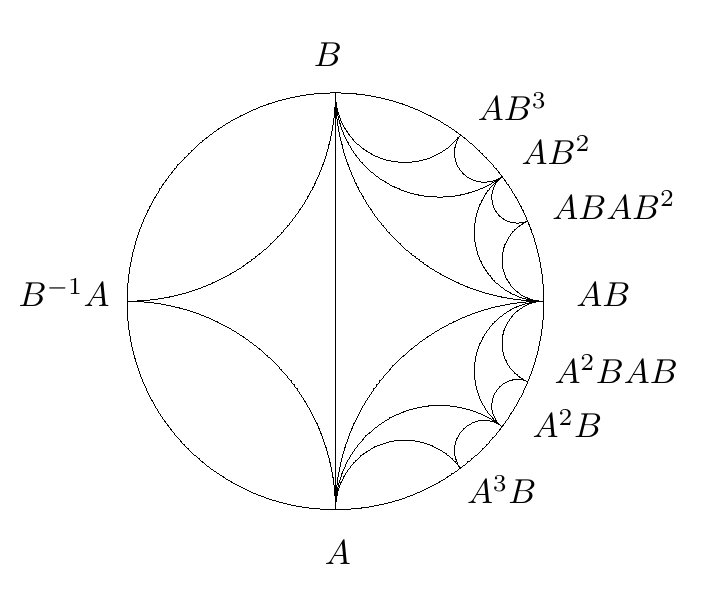}
\caption{The Farey diagram, showing the arrangement of rational numbers on the left with the corresponding  primitive words on the right.}\label{fig:farey}
\end{figure}

 Since the rational points are precisely the images of $\infty$ under $SL(2,\ZZ)$, they correspond bijectively to  the vertices of $\F$.
A pair  $p/q , r/s \in \hat \QQ$ are the endpoints of an edge  if and only if $pr-qs = \pm 1$; such pairs are called \emph{neighbours}.
A triple of points in $\hat \QQ$ are the vertices of a triangle precisely when they are the images of the vertices of the initial triangle $(1/0,0/1,1/1)$; such triples are always of the form 
$(p/q , r/s,(p+r )/( q+s))$ where  $p/q , r/s$ are neighbours.
In other words,  if $p/q , r/s$ are the endpoints of an edge, then the vertex of the triangle on the side away from the centre of the disk is found by `Farey addition' to be $(p+r )/( q+s)$. Starting from $1/0 = -1/0=  \infty$ and $0/1$, all   points in $\hat \QQ$ are obtained recursively  in this way. (Note we need to start with $-1/0=  \infty$  to get the negative fractions on the left side of the left hand diagram in Figure~\ref{fig:farey}.)

The right hand picture in Figure~\ref{fig:farey}
 shows a corresponding arrangement of primitive elements in $F_2$, one in each conjugacy class, starting with initial triple $(A,B, AB)$. Each vertex is labelled by a certain cyclically shortest representative of the corresponding word. Pairs of primitive elements form a generating pair if and only if they are at the endpoints of an edge. Triples at the vertices of a triangle correspond to a generator triple of the form $(U,V, UV)$.  Corresponding to the process of Farey addition, successive words can be found by juxtaposition as indicated on the diagram. Note that for this to work it is important to preserve the order:   if $U, V$ are the endpoints of an edge  with $U$ before $V$ in the anti-clockwise order round the circle, the correct  concatenation is $UV$. Note also that the words on the left side of the diagram involve $B^{-1}A$ corresponding to starting with $\infty = -1/0 $. We denote the particular representative of the conjugacy class corresponding to $p/q \in \hat \QQ$ found by  concatenation by  $W_{p/q}$. Its word length in the generators $A,B$ is a function $F(p/q)$ of $p/q$. A function on $\hat \QQ$ is said to have \emph{Fibonacci growth} if it is comparable with  uniform upper and lower bounds to $F$.

 \begin{figure}[ht] \includegraphics[width=7cm]{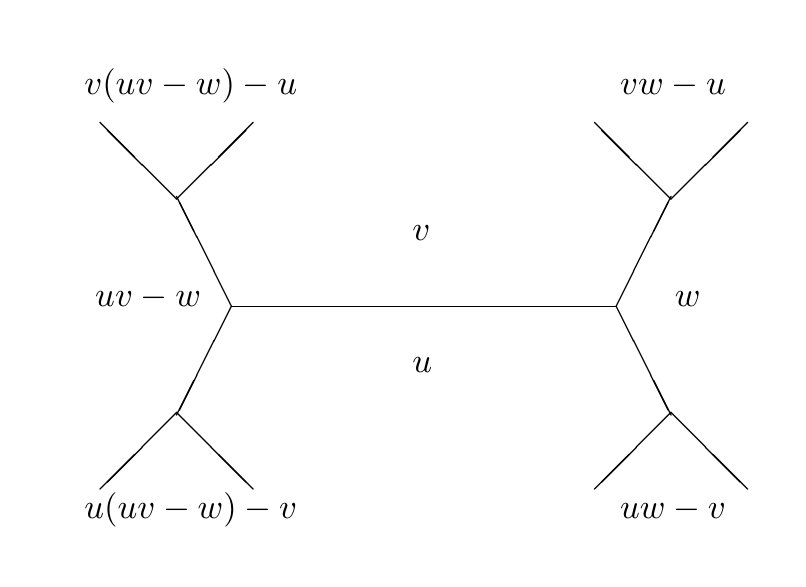}
\vspace{-0.5cm}
\caption{The Markoff tree used to compute traces with an initial triple $(u,v,w)$. }\label{fig:markofftree}
\end{figure}

In this paper we  are largely interested in computing traces of primitive elements.  Following Bowditch~\cite{bow_mar}, these can also be easily computed  by using the  trivalent  tree $\TT$ dual to $\F$, see the left frame of Figure~\ref{fig:farey}, and Figure~\ref{fig:markofftree}. Let $\Omega$ denote the set of complementary regions of $\TT$, abstractly, a complementary region is the closure of a connected component of the complement of $\TT$. As is apparent from Figure~\ref{fig:farey},   there is a bijection between $\Omega$ and the set of vertices of $\F$.   Thus the set $\Omega$ can be identified with conjugacy classes of primitive elements and hence with $\hat \QQ$.

Given a representation $\rho \co  F_2 \to \SL$, each $U \in \Omega$ is labelled by $u=\Tr \rho(U)$, the trace of the corresponding generator, as shown in 
Figure~\ref{fig:markofftree}. Labels on opposite sides of an edge correspond to traces of a generator pair: the three labels round a vertex correspond to a generator triple $(U,V,UV)$.  Crossing an edge adjacent to regions $U,V$  of $\F$ corresponds to changing the  generator triple from $(U,V,UV)$ to  $(U,V,UV^{-1})$.

Suppose that $(U,V,W)$ are 
  the labels of regions round a vertex with $ u = \Tr \rho(U)$, $v = \Tr \rho(V)$, $w = \Tr \rho(W)$. By~\eqref{eqn:commreln}
we have
$u^2+ v^2 + w^2 -uvw = \mu$. By~\eqref{eqn:inverse}, the two vertices  opposite the ends of an edge labelled $(U,V)$ have labels  $w, uv-w$ respectively. More precisely, crossing the $3$ edges of a triangle of $\mathcal F$  gives rise to the three basic moves
 $(u,v, w) \to (u,v, uv-w)$,  $(u,v, w) \to (u,uw-v, w)$,  $(u,v, w) \to (vw-u,v, w)$
which  generates  traces of all possible elements in $\Omega$ (and hence $\mathcal P$).  Note that any of these three moves  leaves $\Tr {[U,V]}$ and hence $\mu$ invariant; in other words, $\mu$ is an invariant of the tree. Bowditch's original paper was mostly confined to the case $\mu = 0$.

In this way, the Markoff tree provides a fast way to compute traces of  elements  in $\P$ starting from an initial triple $(u,v,w)$.   This is illustrated in Figure~\ref{fig:trace2}  with the initial triple $( \sqrt{x+1}, 0, \sqrt{-x+2})$ which is used in Section~\ref{sec:torustree1}. We denote the tree of traces associated to an initial triple $(u,v,w)$ by $\mathbb T_{(u,v,w)}$.
Later we will use a variant of this construction to compute traces of curves on a four pointed sphere, see Section~\ref{sec:traces}.

\begin{figure}[hbt]    \begin{center}\vspace{-.5cm}
 \includegraphics[height=7cm]{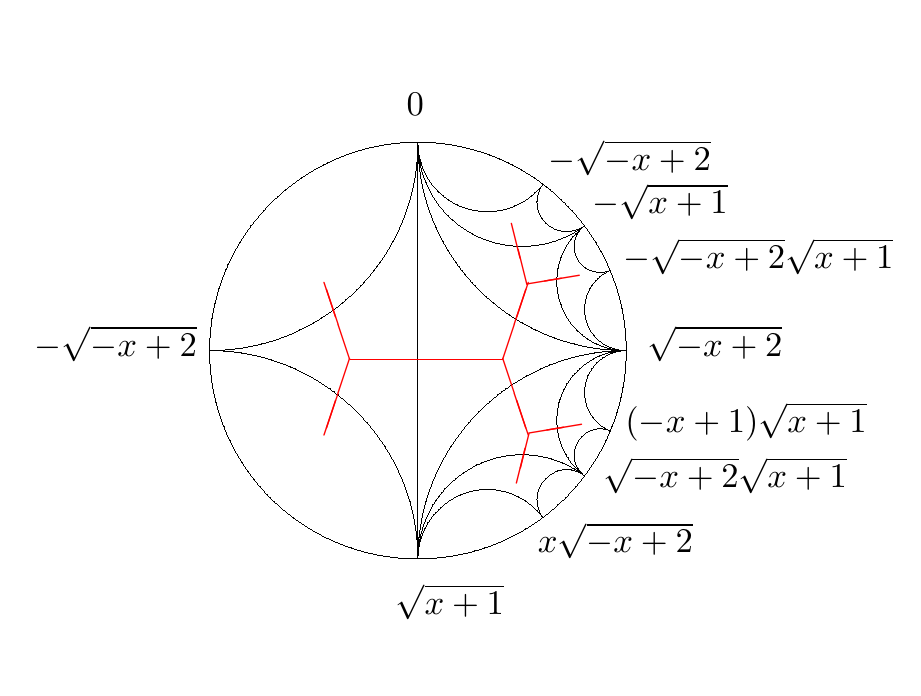}  \vspace{-.8cm}
 \caption{The Farey tessellation used to  compute  traces. See Section~\ref{sec:torustree1} for a discussion of the choice of sign of the square roots.}\label{fig:trace2}
\end{center}
\end{figure}

 \subsection{The Bowditch set} 
\label{sec:tree}

It is convenient to rephrase the above discussion using the terminology introduced in~\cite{bow_mar}.
As above, let $\Omega$ denote the set of complementary regions of the tree $\TT$.
Define a \emph{Markoff map} to be a map $\phi: \Omega \to \CC$ such that $\phi$ satisfies the trace relations~\eqref{eqn:inverse} and~\eqref{eqn:commreln}.   The set of all Markoff maps is denoted $\Phi$. 
 Since traces depend only on conjugacy classes,  
a representation $\rho \co F_2 \to \SL$ defines a Markoff map by setting $\phi(U) = \Tr \rho(U)$ for $U \in \Omega$.  Fixing once and for all an identification of $\Omega$ with $\hat \QQ$ (and recalling that $\Omega$ is identified with conjugacy classes of elements in $\P$), we have $\phi(p/q) = \Tr \rho(W_{p/q}), p/q \in \hat \QQ$, where $W_{p/q}$ is the special word in the conjugacy class corresponding to $p/q \in \Omega$. 

Thus as explained above,  using the trace relations~\eqref{eqn:inverse} and~\eqref{eqn:commreln}, an initial triple $(x,y,z) \in \CC^3$ uniquely determines a Markoff map $\phi = \phi_{x,y,z}$ together with a corresponding labelling of $\TT$.  Conversely a Markoff map $\phi \in \Phi$ determines $(x,y,z)\in \CC^3$ by setting $ x = \phi(0/1),  y = \phi(1/0), z = \phi(1/1)$. In this way, we can identify $\Phi$ with $\CC^3$.
For $\phi \in \Phi$, denote the corresponding tree $\mathbb T_{\phi} = \mathbb T_{(\phi(0/1), \phi(1/0),\phi(1/1))}$.

The Bowditch set $\mathcal B$ is the set of all  $\phi \in \Phi$ with $\mu \neq 4$ which satisfy the following conditions:
\begin{equation}\label{eqn:B1}   \phi (U)  \notin [-2,2]  \ \ \forall  \ U \in \Omega  \ \ \ \mbox {\rm {and}} \end{equation} 
\begin{equation} \label{eqn:B2}   \{U \in \Omega:  | \phi (U)| \leq 2 \}   \ \ \mbox {\rm {is finite}}.
\end{equation}
The Bowditch set $\mathcal B$ is open in $\CC^3$ and ${\rm Out}(F_2)$ acts properly discontinuously on $\mathcal B$. Furthermore, if $\phi \in \mathcal B$, then $\log^+|\phi(U)|={\rm max}\{0, \log|\phi(U)|\}$ has Fibonacci growth on $\Omega$ (see \cite{tan_gen}).

\begin{remark}\label{reducible} \rm{The maps $\phi$ for which  $\mu=4$ correspond  to the reducible representations: our definition above automatically excludes them from $\B$.  For such $\phi$, there are infinitely many $U \in \Omega$ such that $|\phi(U)|<m$  for $m>2$, they can alternatively be excluded from $\B$ by relaxing condition (\ref{eqn:B2}) to the condition  that $\{U \in \Omega:  | \phi (U)| \leq 2+\epsilon \}   $ be finite for any $\epsilon >0$. }
\end{remark}

\subsubsection{Background to the algorithm}\label{sec:algorithm}
Our algorithm for computing which points lie in $\B$ is based on results from~\cite{bow_mar, tan_gen} which we summarise here. We consider only $\phi$ for which $\mu \neq 4$. 
Following Bowditch~\cite{bow_mar}, we orient the edges of $\TT_{\phi}$ in the following way. Suppose that labels of the regions adjacent to some edge $e$ are $u,v$ and the  labels of the two  remaining regions at the two end vertices are $w,t$, see Figure~\ref{fig:markofftree}. From the trace relations, $t = uv-w$.
Orient $e$ by putting an arrow from $t$ to $w$ whenever $|t| > |w|$ and vice versa. If both moduli are equal,  make either choice; if the inequality is strict, say that the edge is \emph{oriented decisively}.

A  \emph{sink region} of $\TT_{\phi}$ is a connected non-empty subtree $T$ such that  the arrow on any edge not in $T$ points towards $T$ decisively. A sink region may consist of a single \emph{sink vertex} $v$ (the three edges adjacent to $v$ point towards $v$) and no edges. Clearly a sink region is not unique: one can always add further vertices and edges around the boundary of $T_{\phi}$.

For any $m \geq 0$ and $\phi \in \Phi$ define $\Omega_{\phi}(m) = \{ U \in \Omega | |\phi( U)| \leq m\}$.
The following  lemmas from~\cite{tan_gen} show that $\Omega_{\phi}(2)$ is connected, and that from any initial vertex not adjacent to regions in $\Omega_{\phi}(2)$, the arrows determine a descending path 
through $\TT$  which either runs into a sink, or meets vertices adjacent to regions in $\Omega_{\phi}(2)$. Furthermore, if $\phi(U)$ takes values away from the exceptional set $E = [-2,2] \cup \{\pm \sqrt{\mu}\} \subset \CC$, then there exists a finite segment of $\partial U$ such that the edges adjacent to $U$ not in this segment are directed towards this segment.

\begin{lemma}[\cite{tan_gen} Lemma 3.7]\label{forkvertex}
Suppose $U,V,W \in \Omega$ meet at a vertex $v$  with the arrows on both the edges adjacent to $U$ pointing away from $v$. Then either $|\phi(U)| \leq 2$ or $\phi(V) = \phi(W) = 0$. \end{lemma}

\begin{corollary}[\cite{tan_gen} Theorem 3.1(2)]\label{connected}
Let $\phi \in \Phi$. Then  $\Omega_{\phi}(2)$ (more generally, $\Omega_{\phi}(m)$ for $m \ge 2$) is connected.
\end{corollary}

\begin{lemma}[\cite{tan_gen} Lemma 3.11 and following comment] \label{infiniteray}
Suppose $\beta$ is an infinite ray consisting of a sequence of edges of $\TT_{\phi}$ all of whose arrows point away from the initial vertex. Then $\beta$ meets at least one region $U \in \Omega$  with $|\phi( U)| < 2$. Furthermore, if the ray does not follow the boundary of a single region, it meets infinitely many  regions with this property.
\end{lemma}

\begin{lemma}[\cite{tan_gen} Lemma 3.20] \label{finiteboundary}
Suppose that $\phi(U) \notin E$ and consider the regions $V_i, i \in \ZZ$ adjacent to $U$ in order round $\dd U$. Then away from a finite subset, the values $|\phi(V_i)|$  are increasing and approach  infinity as $ i \to \infty$ in both directions.   Hence there exists a finite segment of $\partial U$ such that the edges adjacent to $U$ not in this segment are directed towards this segment.
 \end{lemma}

We remark that if $\phi(U)=  \pm \sqrt{\mu}$  and $\sqrt{\mu} \not\in [-2,2]$, then the values of $|\phi(V_i)|$ in Lemma \ref{finiteboundary} approach zero in one direction round $\dd U$ (\cite{tan_gen} Lemma 3.10) and hence $\phi \not\in \B$ since condition (\ref{eqn:B2}) will not be satisfied. Hence, for $\phi \in \B$, $\phi(U) \not\in E$ for all $U \in \Omega$.

The set $\Omega_{\phi}(2)$ can be used to construct a sink region $T$ which is finite if and only if $\phi \in \mathcal B$. Essentially, if $\phi \in \mathcal B$, then $T$ consists of finite segments of the boundaries of the (finite number of) elements of $\Omega_{\phi}(2)$. These are the segments alluded to in Lemma \ref{finiteboundary}; they have to be large enough so the conclusion of the lemma holds, and also to contain all edges adjacent to $U,V$ with $U,V \in \Omega_{\phi}(2)$ so that the union is connected. To do this, an explicit function 
  $H_{\mu}:\CC \rightarrow \RR^+ \cup \{\infty\}$ is constructed (see~\cite{tan_gen} Lemma 3.20, the following remark and Lemma 3.23)   as follows:
\begin{enumerate}
\item If $x \in E$, define $H_{\mu}(x) =\infty$;
\item For $x \not\in E$, let $x=\lambda +\lambda^{-1}$ with $|\lambda|>1$ (note that $|\lambda| \neq 1$ since $x \not\in [-2,2]$). Define
\begin{equation}\label{eqn:Hx} 
H_{\mu}(x) =\max \left\{ 2, \sqrt{\left|\frac{x^2-\mu}{x^2-4}\right|}\frac{2|\lambda|^2}{|\lambda|-1} \right\} . 
\end{equation}
\end{enumerate}
Then $H_{\mu}$ is continuous on $\CC \setminus E$. Now we can define  a specific attracting subtree:

\begin{definition} \label{sinkregion} Given $\phi \in \Phi$, let $T$ be the subset of $\mathbb T_{\phi}$ defined as follows: 
\begin{enumerate}
\item An edge with adjacent regions $U,V$  is in $T$ if and only if either $|\phi(U)| \leq 2$ and $|\phi(V)| \leq H_{\mu}(\phi(U))$, or vice versa. 
\item Any sink vertex is in $T$, as are any vertices which are the end points of two edges in $T$.
\end{enumerate}
\end{definition}

Based on the above lemmas, we have the following theorem (see also the special properties of the function  $H_{\mu}$  and Lemmas 3.21-3.24 in~\cite{tan_gen}).

\begin{theorem} Given $\phi \in \Phi$ (with $\mu \neq 4$), the set  $T$   in Definition \ref{sinkregion} is a non-empty, connected subtree of 
$\TT_{\phi}$. Moreover  $T$ is a sink region for $\TT_{\phi}$, that is, all edges not in $T$ are directed decisively towards $T$. Furthermore, $T$ is finite if and only if $\phi \in \mathcal B$.
\end{theorem}

\subsubsection{The algorithm}\label{sec:algorithm1}

Based on the above discussion, our algorithm to decide whether or not $\phi \in \B$ is as follows. 
\begin{description}
\item [Step 1] Starting at any vertex, follow the direction of decreasing arrows. On reaching  a  sink vertex,
stop. This vertex is in $T$ by Definition~\ref{sinkregion}. If the input is $\B$, then this method always finds a sink vertex in finite time because there is a finite sink region.  Otherwise, the process may not terminate in (pre-specified) finite time and the algorithm is indecisive. 
\item [Step 2]  Assuming a stopping point is found in Step 1, starting from this point, search outwards by a depth first search using Definition~\ref{sinkregion} to identify whether or not an edge is in $T$. This works because of the connectedness of $T$. If this search terminates in (pre-specified) finite time then  $\phi_{x,y,z} \in \B$. Otherwise, the algorithm is indecisive.
\end{description}

Note that  if the starting point is a sink vertex and the three adjacent edges are not in $T$, then $T$ consists of just the sink vertex  by the connectedness of $T$, hence  $\phi_{x,y,z} \in \B$. This occurs for example for the tree $\TT_{(x,x,x)}$ with $x \in (2,3)$.  
\begin{comment}
The algorithm is as follows. It  is justified by Lemmas 3.7, 3.11, 3.21, 3.22 in \cite{tan_gen}.

\begin{description}
\item [Step 1] Starting at any vertex, follow the direction of decreasing arrows. On reaching a vertex with two decreasing arrows (a \emph{fork vertex})
or a vertex with no decreasing arrows (a \emph{sink vertex}),
stop. This vertex is in $T$  since neither a fork nor sink can be contained in the complement of $T$, we are at a vertex in $T$.  
\item [Step 2] We now need to check whether or not $T$ is finite.  Starting from the stopping point found in Step 1, search outwards by a depth first search  using Lemma~\ref{test} to identify whether or not an edge is in $T$. If this search terminates then  $\phi_{(x,y,z)} \in \B$.
Otherwise, the process may not terminate in finite time. Note that in the case where the starting point is a sink, if the three edges adjacent to the sink are not in $T$, then $T$ consists of just the sink and  $\phi_{(x,y,z)} \in \B$. This occurs for example for the tree $\TT_{(x,x,x)}$ with $x \in (2,3)$.
\end{description}

\begin{remark}
 \rm{ In the actual program we used, rather than stopping at a fork vertex, we continue to follow any decreasing arrows until we reach a sink vertex, because in practice, it is easier to write a program which does this.
If the input is $\B$, then this method always finds a sink vertex
in a finite time because there is a finite sink region.
}
\end{remark}
\end{comment}

Figure~\ref{Diagonal-BQ} shows the Bowditch set in the diagonal slice $\Delta$ as determined by this algorithm.

\begin{remark}
 \rm{We do not have an algorithm whose output is  $\phi_{x,y,z} \notin \B$. When $\mu=0$, it was shown in \cite{NT} that if $|\phi(U)|\le 0.5$ for some $U \in \Omega$, then  $\phi_{x,y,z} \notin \B$. Hence in Step 1 above, if $\mu=0$, we can stop when we hit a region satisfying this condition and conclude that  $\phi_{x,y,z} \notin \B$. Using the same methods, a similar upper bound can be found for $\mu$ close to $0$. In particular, there is a neighbourhood of $(0,0,0)$ which is disjoint from $\B$, as clearly illustrated in Figure ~\ref{Diagonal-BQ}. However, as shown in \cite{GMST}, no such universal positive bound exists for all $\mu$: precisely, for any $\epsilon >0$ and $\mu>4$, there exist  $\phi \in \B_{\mu}$ and $U \in \Omega$ such that $|\phi(U)|<\epsilon$.  Another issue is that the sink region may be extremely large so may not be detected in a program with a given finite number of steps, this occurs when we approach the boundary of $\B$. Thus the  algorithm is not completely decisive although it appears to give nice results. In particular, there may be false negatives; however points which are determined to be in $\B$ are correctly marked. }
\end{remark}

  \begin{figure}[ht]
  \includegraphics[width=7cm]{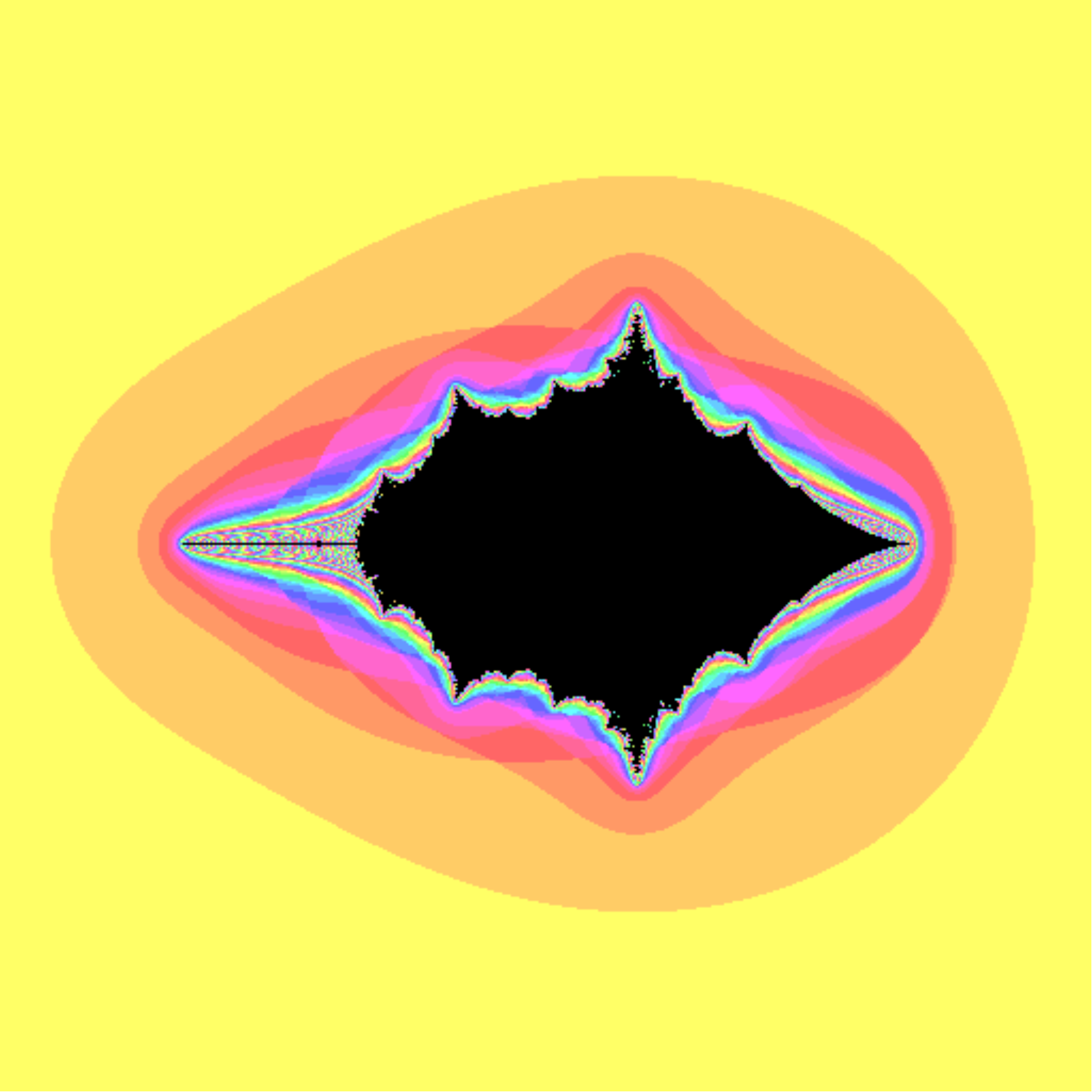}
\caption{The Bowditch set  $\B$ for the Markoff maps $\phi_{(x,x,x)}$, plotted in the $x$-plane. The coloured (grey) points are in  $\B$ and the black ones are outside. The colours (shades of grey) indicate the size of the sink region $T$.} \label{Diagonal-BQ} \end{figure}

\section{Groups, manifolds, symmetries and quotients} \label{sec:basicconfig}

In this section we detail a construction which allows us conveniently to exploit the three-fold symmetry of groups in the diagonal slice $\Delta$. 
As is well known, if the image of a representation $\rho \co F_2 \to \SL$ is free and discrete  then $\HH^3/\rho(F_2)$ is a genus two handlebody $\H$,  see~\cite{hempel} Theorem 5.2. (Note that a hyperbolic $3$-manifold is irreducible, hence prime, and that $\pi_2(M) = 0$.)
Rather than working with $\H$, however, it is much easier to work with the quotient $\S$ of $\H$ by the order $3$-symmetry $\O$ corresponding to cyclic permutation of the parameters. We also introduce  a  commensurable orbifold $\T$ with a torus boundary  $ \dd \T$.  

Both $\S = \H /\O$ and $\T$  surject to a $3$-orbifold $\U$ with fundamental group  a so-called \emph{$(P,Q,R)$-group}. Its boundary  $\dd \U$ is a sphere with three order $2$ and one order $3$ cone points. A similar construction has been used extensively by Akiyoshi et al, see for example~\cite{akiyoshi}, and is the basis of Wada's program OPTi, hence was convenient for our computations.   
In this section we explain these constructions in detail,  using them to find explicit representations of all four groups.

\subsection{The handlebody and related orbifolds} \label{sec:mainhandlebody}

The symmetric handlebody $\H$ can be thought of as made by gluing two solid pairs of pants each with order $3$-symmetry. More precisely, take a $3$-ball  and remove three open disks from the boundary, placed so as to have order $3$ rotational symmetry. Gluing two such balls along the open disks produces a handlebody $\H$ with the required order three symmetry $\O$. Rather than write down a suitably symmetric representation of $\pi_1(\H)$ directly, we consider first the quotient orbifold $\S = \H /\O$.  As will be justified in retrospect when we have identified the representations explicitly, this is a ball with  two cone axes around each of which the angle is $2\pi/3$. Its boundary $\dd \S$ is a sphere $\Sigma_{0;3,3,3,3}$ with $4$ order $3$ cone points.  We will call $\S$ the \emph{large coned ball}. 

The ball $\S$ has a further order $4$ symmetry group. 
Consider  the two cone axes  which form the singular locus of $\S$, together with their common perpendicular $C$. This configuration is invariant under the $\pi$-rotation about $C$, and also under $\pi$-rotations about a unique pair of orthogonal 
 lines  on the plane orthogonal to $C$ passing through its midpoint $O$, see Section~\ref{sec:largeball}.  Denoting these latter rotations $P,Q$, the $\pi$-rotation about $C$ is $PQ$ and the entire configuration is invariant under 
 $\langle P,Q \rangle = \ZZ_2 \times \ZZ_2$. Thus we obtain a further quotient  orbifold $\U = \S/(\ZZ_2 \times \ZZ_2)$, also topologically a ball,  which we call  the \emph{small coned ball}. 
The singular locus of $\U $ is as follows. Let $\bar O$ and $\bar  E$ be the images in $\U$ of the midpoint $O$ of $C$ and the point where $C$ meets $\Ax K$ respectively, where $K$ is one of the two order three rotations. Let $\bar C$ be the image of $C$, so that $\bar C$ is a line from $\bar O$ to $\bar E$.
 From  $\bar  O$ emanate three mutually orthogonal   lines  corresponding  to the order $2$ axes of $P,Q$ and $PQ$. One of  these is the line  $\bar C$ corresponding to $PQ$ which ends at $\bar  E$. From $\bar  E$ also emanates an order $3$ singular line, the projection of $\Ax K$, perpendicular to $\bar C$. The boundary $\dd \U$ is a sphere $\Sigma_{0;2,2,2,3}$ with $3$ cone points of order $2$  and one of order $3$. The order $3$ cone point is the end point of the order 3 singular line and the order $2$ cone points are the endpoints on $\dd U$ of the axes of $P, Q$ and a third involution $R$ defined below.

Finally, there is a double cover of the small coned ball $\U$ by an orbifold $\T$ which is topologically a solid  torus. Its
 boundary  is a torus $  \dd \cal T$ with a single cone point of angle $4 \pi/3$. Just as the quotient of a once punctured torus $  \Sigma_{1;\infty}$ by the hypelliptic involution is the surface $\Sigma_{0;2,2,2,\infty}$, so the quotient of    $ \dd \cal T $  by the hypelliptic involution $\iota$ is the surface $\dd \U = \Sigma_{0;2,2,2,3}$.   The involution $\iota$ extends to an involution, also denoted $\iota$, of $\cal T$ such that $\cal T/\iota = \U$.

The group $\pi_1(\U)$  is generated by  $(P,Q, K)$. We can replace $K$ 
by a further involution $R$ such that $RQP = K$. 
 To do this,  let  $R$ be an order $2$ rotation about an axis contained in the plane through $E$ orthogonal to $ \Ax K$, such that the axis makes an angle $\pi/3$ with $C$. (We will fix orientations more precisely below.) Then $R (QP)$ is a $2\pi/3$-rotation about $ \Ax K$, in other words, provided orientations have been chosen correctly,  we can identify $\pi_1(\U) $ with a group $  \langle P,Q,R | P^2=Q^2=R^2= (RPQ)^3  = -\rm{id}, PQ=-QP \rangle \subset \SL$. (For a discussion on the choice of signs, see Remark~\ref{signchoices} below.)

In~\cite{akiyoshi} and other papers by the same authors,  groups generated by three involutions $P,Q,R$
with $R QP$ parabolic, 
are used as a convenient way of parameterizing representations of once punctured tori, where the
 torus in question is now a two-fold cover  of  the orbifold  
 with fundamental group $\langle P,Q,R \rangle$ with quotient induced by the hyperelliptic involution.
  A small modification of their parameterization allows us to write down a convenient  general form for a representation $\pi_1(\U) \to \SL$, from which we obtain explicit representations of $\pi_1(\H), \pi_1(\S)$ and $\pi_1(\T)$. This we do in the next section.

 \subsection{The basic configuration and the small coned ball}\label{sec:basicconfig1}
 
 We start with a general construction for representations $\pi_1(\U) \to \SL$, that is, of subgroups $ \langle P,Q,R | P^2=Q^2=R^2= (RPQ)^3  = -\rm{id}, PQ=-QP \rangle \subset \SL$.  For convenience we refer to such a group (or its image in $\PSL$) as a \emph{$(P,Q,R)$-group}.

We will make our calculations using  \emph{line matrices} following \cite{Fenchel}. Note this  will define representations into $\SL$, thus fixing the signs of traces.
Let $u,u' \in \Chat$, and denote the oriented line from $u$ to $u'$ by $[u,u']$.
The associated line matrix $M([u,u'])\in\SL $ is a matrix which induces an order two rotation about $[u,u']$
and such that $
M([u,u'])^2 = -\rm{id}$, so that in particular
 $$M([0,\infty])  = \begin{pmatrix} i & 0 \\ 0 & -i \end{pmatrix}.$$
 By \cite{Fenchel}, p. 64, equation (1), we have, if $u,u' \in \CC$:
$$
M([u,u']) = \frac{i}{u-u'}
\begin{pmatrix} u+u' & -2uu' \\ 2 & -u-u' \end{pmatrix}.
$$

The representation we require is derived from 
a basic configuration shown in Figure~\ref{fig:basicconfig}. 
It depends on a single parameter $ \z \in \CC$ which we will relate to the original parameter $x$ in \ref{sec:solidtorus} below.

 \begin{figure}[hbt] 
\includegraphics[height=4cm]{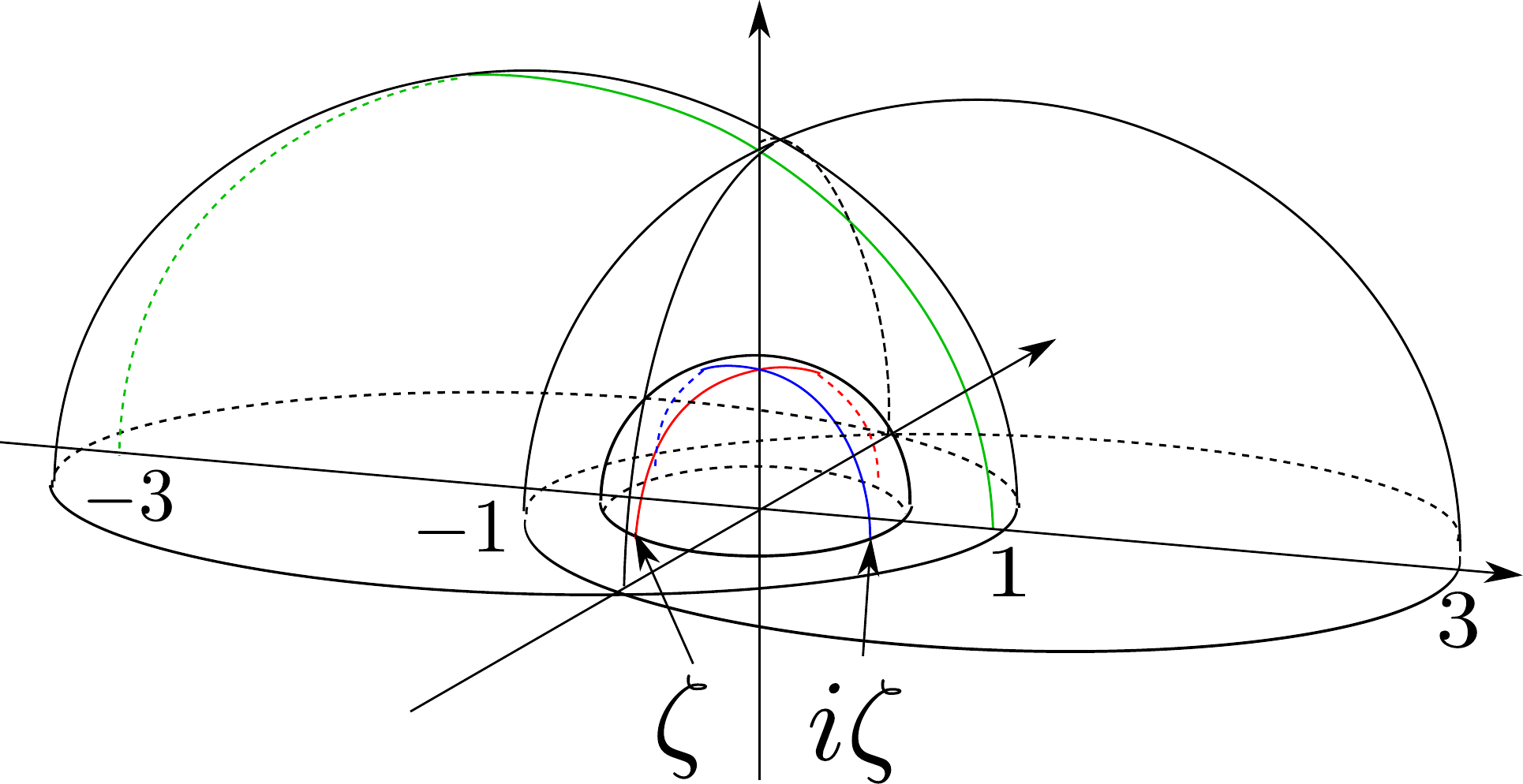} 
 \caption{The basic configuration for the $(P,Q,R)$-group $\pi_1(\U)$.}\label{fig:basicconfig}
\end{figure}

Let $\z\in\mathbb C$ and $P, Q, R\in\SL $ be $\pi$-rotations
about the oriented lines $[\z, -\z]$, $[i\z, -i\z]$ and $[1,-3]$, respectively. By construction $P^2=Q^2= R^2 = -\rm {id}$. Moreover $ \Ax P$ and $ \Ax Q$ intersect at the point $|\z|j \in \HH^3$ on the hemisphere of radius $|\z|$ centre $0 \in \CC$, where $ z + tj$ represents the point at height $t>0$ above $z\in \CC$ in the upper half space model of $\HH^3$. Thus $PQ = -QP$ and $PQ$ is an order $2$ rotation about the vertical axis $0 + tj, t>0$.

Let $V$ be the vertical plane  above the real axis  in $\HH^3$. Note that  the oriented axes of the order two rotations $PQ$ and $R$ both lie in  $V$,  intersecting in the point $\sqrt 3 j $   at angle $\pi/3$.
The line $[-\sqrt{3}i, \sqrt{3}i]$ passes through this point and is orthogonal to $V$.
It follows that $RPQ = -RQP$  is anti-clockwise rotation through $2\pi/3$ about the line 
$[-\sqrt{3}i, \sqrt{3}i]$.  Using line matrices as above, we can now easily write down the 
corresponding representation in $\SL$:

\begin{align*}
P & = M([\z, -\z])
    = \frac{i}{2\z} \begin{pmatrix} 0 & 2\z^2 \\ 2 & 0 \end{pmatrix}
    = \begin{pmatrix} 0 & i \z \\ i/\z & 0\end{pmatrix}\\
Q & = M([i\z, -i\z])
    = \frac{1}{2\z} \begin{pmatrix} 0 & -2\z^2 \\ 2 & 0 \end{pmatrix}
    = \begin{pmatrix} 0 & -\z \\ 1/\z & 0 \end{pmatrix} \\
R & = M([1, -3])
    = \frac{i}{4} \begin{pmatrix} -2 & 6 \\ 2 & 2 \end{pmatrix}
    = \begin{pmatrix} -i/2 & 3i/2 \\ i/2 & i/2 \end{pmatrix}.
\end{align*}
Let $K = RQP$.  Then,
$$
K = \begin{pmatrix} -1/2 & -3/2 \\ 1/2 & -1/2 \end{pmatrix}, \quad
K^3 = \begin{pmatrix}1 & 0 \\ 0 & 1 \end{pmatrix},
$$
 so that as expected, $K$ is a anticlockwise rotation about $[-\sqrt{3}i, \sqrt{3}i]$ by $2\pi/3$.

Note that as matrices in $\SL$, $P^2 = Q^2 = -\rm {id}$ and $PQ = -QP$. As isometries of $\HH^3$, the signs are irrelevant. We could have chosen $K = RPQ$ in which case $K^3 = -\rm {id}$ but see Remark~\ref{signchoices} below.
We denote the group generated by $P,Q,R$ by $G_{\U}(\z)$ and the corresponding representation $\pi_1(\U) \to \SL$ by $\rho_{\U}(\xi)$.

 \subsubsection{The large coned ball $\S$}\label{sec:largeball}

To relate $\pi_1(\U)$ to $\pi_1(\S)$, start with two oriented axes $A_0, A_1$ about each of which we have order $3$ anticlockwise rotations $K_0,K_1$, measured  with respect to the orientation of the axes. Let $C$ denoted the common perpendicular between $A_0$ and $ A_1$, oriented from $A_0$ to $A_1$. We denote this configuration, which is clearly well defined up to isometry, by $\C \F$.
As described in
\ref{sec:mainhandlebody}, 
$\C \F$ has a further $\ZZ_2 \times \ZZ_2$ group of symmetries generated by 
 the $\pi$-rotations $P,Q$ with axes through the  mid-point of $C$: precisely, let $\Pi$ be the plane through  the mid-point of $C$ and orthogonal to $C$. Then 
 the axes of $P,Q$ are the two lines in $\Pi$ which bisect the angles between the projections of   $\Ax K_0, \Ax  K_1 $ onto $\Pi$, chosen so that the angle bisected by  $ \Ax P$ is that between the projection of the lines  $\Ax K_0, \Ax  K_1 $  with the same (say outward) orientation.
 
This choice of $P$ ensures that $PK_0P^{-1} = K_1$ while $QK_0Q^{-1} = K_1^{-1}$. 
Also $PQ$ is the order $2$ rotation about $C$, and $PQK_i Q^{-1}P^{-1} = K_i^{-1}, i = 0,1$.
As in Section~\ref{sec:mainhandlebody}, 
$\U = \S/(\ZZ_2 \times \ZZ_2)$ and we can take $\pi_1(\U)$ to be the $(P,Q,R)$-group defined in~Section~\ref{sec:basicconfig1}. 
In terms of  $(P,Q,R)$, the 
generators of $\pi_1(\S)$ are $K_0 =  RQP, K_1 = PK_0P^{-1}$.
Thus
\begin{equation*}
K_0 =  -\begin{pmatrix} 1/2 & 3/2 \\ -1/2 & 1/2 \end{pmatrix}, \quad
K_1 =- \begin{pmatrix} 1/2 & -\z^2/2 \\ 3/2\z^2 & 1/2 \end{pmatrix}.
\end{equation*}

In terms of generators for $ \pi_1(\dd \S)$, we have additionally 
$ K_2 = QK_0Q^{-1}, K_3 = RK_0R^{-1}$ where 
\begin{equation*}
K_2 =  -\begin{pmatrix} 1/2 & \z^2/2 \\ -3/2\z^2 & 1/2 \end{pmatrix}, \quad 
K_3 = -\begin{pmatrix} 1/2 & -3/2 \\ 1/2 & 1/2 \end{pmatrix}
\end{equation*}
so that $K_0 K_3 K_1 K_2 = \rm{id}$.

 We denote  the group generated by $K_0,K_1$ by $G_{\S}(\z)$ and the corresponding representation $\pi_1(\S) \to \SL$ by $\rho_{\S}(\xi)$.
From now on, we frequently drop the subscript and refer to $K_0$ as $K$.

\subsubsection{\textbf{The handlebody $\H$}}\label{sec:handlebody}

Observe that the generator $X \in \pi_1(\H)$ projects to the loop $K_0K_1$ in $\H /\O$. (This latter is  a loop in $\dd \H /\O$ which separates one of each pair of the cone points of $K_0,K_1$ from the other pair.)
We arrange that the action of $\O$ is induced by conjugation by $K_0^{-1} = K^{-1}$, so the generators of $\pi_1(\H)$ can be written in terms of the generators of $\pi_1(\S)$ as $X = K_0 K_1$, $Y = K_0^{-1}XK_0 = K_1 K_0$. Thus we have:
$$
K^{-1} X K = Y, \quad
K^{-1} Y K = (XY)^{-1}, \quad
K^{-1} (XY)^{-1} K = X.
$$
Using the formulae from the previous section, this gives
$$
X = \begin{pmatrix} 
    \frac{9}{4\z^2} + \frac{1}{4} &
    -\frac{\z^2}{4} + \frac{3}{4} \\
    \frac{3}{4\z^2} - \frac{1}{4} &
    \frac{\z^2}{4} + \frac{1}{4} 
    \end{pmatrix},\quad
Y = \begin{pmatrix} 
    \frac{\z^2}{4} + \frac{1}{4} &
    -\frac{\z^2}{4} + \frac{3}{4} \\
    \frac{3}{4\z^2} - \frac{1}{4} &
    \frac{9}{4\z^2} + \frac{1}{4}
    \end{pmatrix}.
$$
In particular this reveals the relation between the parameter $\z$ and $x$:
\begin{equation}\label{eqn:paramreln}
x= \tr X = \tr Y = \tr XY = \frac{\z^2}{4}+\frac{9}{4\z^2}+\frac{1}{2}.
\end{equation}
We denote the group generated by $X,Y$ by $G_{\H}(\z)$ and the corresponding representation $\pi_1(\H) \to \SL$ by $\rho_{\H}(\z)$; we explain in Section~\ref{sec:whichparameter}  why up to conjugation $\rho_{\H}(\z)$ in fact depends only on $x$.
 
 \begin{remark} \label{signchoices}   \rm{In the above discussion, we made choices of sign so that $K^3 = \mbox{\rm{id}}, X = K_0K_1$ (where $K=K_0$ as above).   To compute the discreteness locus of a family of representations only  requires looking in $\PSL$, however for computations involving  traces  we need a lift to $\SL$.
 
 By~\cite{Culler},  any $\PSL$  representation of a Kleinian group can be lifted  to $\SL$ provided there are no elements of order $2$; in particular  this applies to $\PSL$ representations of $\pi_1(\S)$ and $ \pi_1(\H) $.  
 Since the product of the three generating loops corresponding to $X,Y, Z$ is the identity in $\pi_1(\H)$,  we should make a choice of lift in which $XYZ = \rm{id}$ in $\SL$.  
 We could choose the element $K$ which represents the  $3$-fold symmetry $\O$ to be  such that either  $K^3 =   \rm{id}$ or 
 $K^3 = -\rm{id}$; however since we intend to work with   representations of $\pi_1(\S) \to \SL$ we should make the choice $K^3 =   \rm{id}$ because in the quotient orbifold $\S$, $K$ corresponds to a loop round  an order $3$ cone axis.

In the representation we have written down we achieve    $K^3 = \rm{id}$ with  the choice
  $K =  RQP = \begin{pmatrix} -1/2 & -3/2 \\ 1/2 & -1/2 \end{pmatrix}$. 
It is easy to check that taking $K^3 = \rm{id}$, if we let $X = K_0K_1$ we get $XYZ = \rm {id}$ as required, but if we choose $X = -K_0K_1$ we get $XYZ = -\rm {id}$ which is wrong.
}
 \end{remark}

 \subsubsection{\textbf{The singular solid torus $\T$.}}\label{sec:solidtorus}
 
Finally we discuss the associated singular solid torus  $\T$, which is constructed in a standard way from the $(P,Q,R)$-group. We do not logically need to use $\T$ in our further development, however as explained in Section~\ref{sec:torustree},   in practice we used $\T$ for computations, moreover the interpretation of the problem in the more familiar setting of a torus with a cone point may be helpful.

The boundary $  \dd \U$  is a sphere with $4$ cone points  $x_P, x_Q,x_R$ and $ x_K$ corresponding to $P,Q, R$ and $K = RQP$.  Thus we can take as generators of $\pi_1(\T)$ the loop $B=PQ$ separating $x_P, x_Q$ from $x_R, x_K$, and the loop $A = RQ$ separating $x_R, x_Q$ from $x_P, x_K$. Since $P,Q$ have a common fixed point, $B$ is an order $2$ elliptic, while since the axes of $R,Q$ are (generically) disjoint, $A$ is a loxodromic whose axis extends the common perpendicular to $ \Ax R$ and $ \Ax Q$.

Using the formulae above for the $(P,Q,R)$-group we compute:
\begin{equation*}
RQ = A    = \begin{pmatrix} 3i/2\z & i\z/2 \\ i/2\z & -i\z/2 \end{pmatrix},  \quad
PQ= B  = \begin{pmatrix} i & 0 \\ 0 & -i\end{pmatrix},
\end{equation*} \quad
so that \begin{equation}\label{eqn:torustraces}
\tr A   = \frac{3i}{2\z} - \frac{i\z}{2}, \quad
\tr B =  0, \quad
\tr AB = -\frac{\z}{2} - \frac{3}{2\z}. 
\end{equation}

Note that $AB = RP$ and $A^2 = -K_0K_1, B^2=-\rm{id}$.  
We also deduce that 
\begin{equation*}
ABA^{-1} B^{-1} = [A, B] = \begin{pmatrix} 1/2 & -3/2 \\ 1/2 & 1/2 \end{pmatrix},  \\ \ 
\mbox{{\rm so that }}
\Tr {[A,B] }= 1.
\end{equation*}

Note that $\Tr A \in [-2,2]$ if and only if $|\z| = \sqrt 3$ or $\z = it$ with 
$ 1 \leq |t| \leq \sqrt 3$,
 justifying the above remark that generically $A$ is loxodromic. 
Note also that $A^2 = -K_0K_1 $   is consistent with the direct computation using \eqref{eqn:torustraces}
 that $\Tr (A^2) = -( \frac{\z^2}{4}+\frac{9}{4\z^2}+\frac{1}{2})$.  Also note that $ [A, B] = -K^2$, so that the commutator is rotation by $ 4\pi/3$ about $ \Ax K$.  Since $\Tr K^2 = (\Tr K)^2 -2$ we find also that $\Tr {[A,B]} =  1$ independently  of the choice of sign for $K$. This is consistent with  $\Tr {[A,B]} = -2 \cos (2\pi/3)$,  the sign being negative by analogy  with the well known fact that for any representation of a once punctured torus group for which the commutator is parabolic, we have  $\Tr  {[A,B]}= -2 $.
 
 We denote  the group generated by $A,B$ by $G_{\T}(\z)$ and the corresponding representation $\pi_1(\T) \to \SL$ by $\rho_{\T}(\xi)$.
 
 \begin{remark} \label{signchoices1}   \rm{Once again there are questions of sign which this time are a little more subtle. If   $\a \in \PSL$ corresponds to an element of order $2$ in $\pi_1(M)$, then the corresponding representation cannot be lifted to $\SL$, because for non-trivial $\a \in \SL$, necessarily  $\a^2 = -\rm{id}$, see~\cite{kra} and \cite{Culler}. 
Since in $\pi_1(\T)$ the element $B^2$ is trivial,  a $\PSL$ representation of  $\pi_1(\T)$ cannot be lifted to $\SL$.
Nevertheless,  we can as above write down a group in  $\SL$ which projects to a $\PSL$ representation for $\pi_1(\T)$. See Section~\ref{sec:torustree} for further discussion on this point.
} \end{remark}

 \subsubsection{More on the configuration for the large coned ball $\S$}\label{sec:morelargeball}
 
 The relation~\eqref{eqn:paramreln} can be given a geometrical interpretation in terms of the perpendicular distance between the axes of $K_0,K_1$ which sheds light on the symmetries of  the configuration  $\C \F$ in Section~\ref{sec:largeball}.
To measure complex distance, we use the conventions spelled out in detail in~\cite{ser-wolp} Section 2.1. The signed complex distance $ {\bf d}_{ \a} (L_1,L_2)$ between two oriented lines $L_1,L_2$ along their oriented common perpendicular $\a$ is defined as follows. 
The signed real distance $  d_{  \a} (L_1,L_2)$  is the positive real hyperbolic distance between $L_1,L_2$  if $\alpha$ is oriented from $L_1$ to $L_2$ and its negative otherwise. Let ${\bf v_i}, i=1,2$ be unit vectors to $L_i$ at the points $L_i \cap \a$ and let $\bf w_1$ be the parallel translate of $\bf v_1$ along $\a$ to the point $ \a \cap L_2$.
Then $  {\bf d}_{ \a} (L_1,L_2)= \delta_{\bf \a} (L_1,L_2)  + i\theta$ where $\theta$ is the angle, mod $2\pi i$, from $\bf w_1$ to $\bf v_2$ measured anticlockwise  in the plane spanned by  $\bf w_1$ to $\bf v_2$ and oriented by $\a$.

Let 
$\sigma$ be the  signed complex distance from the \emph{oriented} axis $\Ax  K_0 $ to 
 the oriented axis $ \Ax  K_1 $, measured along the common perpendicular $C$  oriented  from $\Ax  K_0 $ to 
$ \Ax  K_1$. Then $\Ax  K_0 ,  \Ax  K_1$ together with  $\Ax  K_0K_1 $ form the alternate sides of a right angled skew hexagon whose other three sides are the common perpendiculars between the three axes taken in pairs. The cosine formula gives $\sigma$  in terms of the complex half translation lengths $\lambda_0, \lambda_1, \lambda_2$  of $K_0, K_1$ and $K_0K_1$ respectively. To get the sides oriented consistently round the hexagon we have to reverse the orientation of $ \Ax  K_0$  so that the complex distance $\s$  should be replaced by $ \s' = \s+ i \pi$ and $\lambda_0$  by $ \lambda'_0  = -\lambda_0 $, see~\cite{ser-wolp}, so the   formula  gives 
$$ \cosh \sigma' = \frac{\cosh \lambda_2 - \cosh  \lambda'_0 \cosh  \lambda_1}{ \sinh  \lambda'_0 \sinh  \lambda_1} .$$
 
As in~\ref{sec:handlebody}, we have $X  = K_0K_1$ so
 $ x = \tr K_0K_1 = 2 \cosh \lambda_2$
while for $i = 0,1$ we have $ \cosh \lambda_i = \cos 2 \pi/3 = -1/2$ and $\sinh \lambda_i = i\sin 2 \pi/3 = i \sqrt 3/2$.  (Note that since $K_0, K_1$ are conjugate we should take $\lambda_0= \lambda_1$ so the possible additive ambiguity of $i\pi$  in the definition of the $\lambda_i$  does not change the resulting equation.) Substituting, we find 
\begin{equation}\label{eqn:cxdist}
 -\cosh \sigma = \frac{x/2 - 1/4}{( \sqrt {3}/2)^2} = \frac{2x-1}{3}.\end{equation}

 We can also relate $\s$ directly to our parameter $\z$.  By construction $ \Ax  K_0$ is the oriented line
 $[-\sqrt 3 i, \sqrt 3i]$, while $K_1 = PK_0P^{-1}$ so that $ \Ax  K_1$ is the oriented line
 $[  i\z^2/\sqrt  3,  -i\z^2/\sqrt  3]$ and   $C$ is the oriented line from $\infty$ to $0$. 
 Thus the real part of the hyperbolic distance from $ \Ax  K_0$ to $ \Ax  K_1$ is $2\log \sqrt 3 /|\z|$ and the anticlockwise
 angle, measured in the plane oriented \emph{downwards} along the vertical axis $C$, is $-(\pi + 2\arg \z) $.  
 Hence $$\sigma = 2\log \sqrt 3 /|\z| - 2i\arg \z - i\pi= 2 \log \frac{\sqrt 3}{\z}.$$ 
  Comparing  to~\eqref{eqn:cxdist}, we find
   $$ \biggl [  \biggl(\frac{\sqrt 3}{\z}\biggr)^2 + \biggl(\frac{\z}  {\sqrt 3}\biggr)^2\biggr ] = 2 \cosh (\sigma + i \pi) =\frac{2(2x-1)}{3}$$
     or 
   \begin{equation}\label{eqn:paramreln1}  x -1/2 = \frac{3}{4}   \biggl [  \biggl(\frac{\sqrt 3}{\z}\biggr)^2 + \biggl(\frac{\z}  {\sqrt 3}\biggr)^2\biggr ]   \end{equation}
 recovering and  giving a more satisfactory geometrical meaning to~\eqref{eqn:paramreln}.

 \subsubsection{Dependence on $x$ versus $\xi$.}\label{sec:whichparameter}
 It is not  perhaps  immediately obvious why the groups $G_{\S}(\xi), G_{\H}(\xi)$ as defined above depend up to conjugation only on our original parameter $x$. This is clarified by the above discussion, because up to conjugation $G_{\S}(\xi)$ depends only on the configuration $\C\F$ and hence on $\sigma$ which is related to $x$ as in~\eqref{eqn:cxdist}. An alternative way to see this is the discussion on computing traces in Section~\ref{sec:traces}. Thus from now on, we shall alternatively write 
$G_{\S}(x), G_{\H}(x)$ in place of $G_{\S}(\xi), G_{\H}(\xi)$.

   \subsubsection{Symmetries}\label{sec:symmetries}
   The discussion in~\ref{sec:morelargeball} gives  insight into various symmetries of the parameters $x$ and $\z$.
Equation~\eqref{eqn:paramreln} shows that the map $ \z \mapsto x$ is a $4$-fold covering with branch points at $ \z = \pm\sqrt{3},  \z = \pm i\sqrt{3}$  and $z = 0,\infty$.
Correspondingly, we have a Klein  $4$-group $\ZZ_2 \times \ZZ_2$ of symmetries which change $\z$ but not $x$:
\begin{enumerate}
\item  replacing $\z$ by $-\z$ leaves the basic construction unchanged but the line matrices defining $P,Q$ change sign.

\item  replacing $\z$ by $-3/\z$  is an order $2$ rotation about the axis  $[-\sqrt 3i, \sqrt 3 i]$. This fixes $K_0$ and moves $K_1$ into a position on the opposite side of $K_0$ along the vertical line $C$. This changes nothing other than the position we choose for the basic configuration in Section~\ref{sec:basicconfig1}. Note however that 
the line matrices defining $P,Q$ change sign.
\end{enumerate}

There is also a symmetry which changes $x$ as well as $\z$. 
 Say we fix the orientation of one of the two axes $\Ax K_0,  \Ax  K_1 $ while reversing the other. On the level of the  configuration $\C \F$ from~\ref{sec:largeball},  this interchanges $P$ and $Q$. Since  $PK_0P^{-1} = K_1$ while $QK_0Q^{-1} = K_1^{-1}$, 
 this is equivalent to  fixing the orientation of one of the two axes $ \Ax  K_0 ,  \Ax  K_1 $ while reversing the other.   This symmetry interchanges the marked group $P,Q, R$ with the marked group $Q,P, R$, so that one group is discrete if and only if so is the other.
  In terms of our parameters,  the complex distance $\sigma$ between the axes changes to $ \sigma + i\pi$, so that $\cosh \sigma \mapsto - \cosh \sigma$ giving the symmetry $(x -1/2) \mapsto -(x -1/2)$ of  \eqref{eqn:cxdist}. Note that the diagonal slice of the Bowditch set $\Delta \cap \B$ does not possess this symmetry.
 Interchanging $P$ and $Q$  is  induced by the map $ \z \mapsto i\z$; more precisely this map sends $P$ to $Q$ and $Q$ to $-P$.  This clearly  induces the same symmetry in equation \eqref{eqn:paramreln}.  Note that by the definition, in this symmetry $R$ remains unchanged.

 On  the level of the torus group $\pi_1(\T)$, we have by definition  $RQ= A$, $PQ=B$ so that $AB = RP$. Thus sending $P$ to $Q$ and $Q$ to $-P$ while fixing $R$ sends $B$ to $-B$ and  $A$ to $-AB$. (Recall that on the level of matrices, $PQ = -QP$.) The symmetry should therefore replace the trace triple $(\Tr A, \Tr B, \Tr AB)$  by the triple $(-\Tr AB, -\Tr B, \Tr A)$.
It is easily checked from~\eqref{eqn:torustraces} that this is exactly the change effected by $ \z \mapsto i\z$.

Finally, we have the  symmetry of complex conjugation induced by $ x \to \bar x$ or equivalently $ \z \mapsto \bar \z$. This sends $\sigma \mapsto \bar \sigma$ thus replacing  $G_{\H}(x)$ by a conjugate group in which the distance between $ \Ax  K_0 $ and $ \Ax  K_1 $  is unchanged but the angle measured along their common perpendicular changes sign. Clearly these are different groups but one is discrete if and only if the same is true of the other. 
 
 The diagonal slice of the Bowditch set obviously also enjoys the symmetry by conjugation, however, that is its only symmetry. In particular   $(x,x,x) \to (-x,-x,-x)$  is not  a symmetry and  the corresponding $\SL$ representations project to different representations of $F_2$ into $\PSL$. This is because any two distinct 
lifts of a  representation from $\PSL$ to $\SL$  differ by multiplying  exactly two of the parameters $x,y,z$ by $-1$.  
The allowed replacement $X \to -X$ and $Y \to -Y$
gives the group $(-x,-x,x)$ with parameters which are not in the diagonal slice $\Delta$.

 The symmetries can be seen in our plots by comparing 
 Figure~\ref{Diagonal-BQ}, the Bowditch set for the triple $\phi_{(x,x,x)}$  in the $x$-plane,  with the right hand frame of Figure~\ref{fig:BQ-sets-comparison}, which shows the same set in the 
 $\z$-plane.  Note the symmetry of complex conjugation in both pictures. In addition, Figure~\ref{fig:BQ-sets-comparison} is invariant under the maps $ \z \mapsto -\z$ and $ \z \mapsto -3/\z$, neither of which are seen in  Figure~\ref{Diagonal-BQ}. Thus the upper half plane in
Figure~\ref{fig:BQ-sets-comparison} is a $4$-fold covering of the upper half plane in Figure~\ref{Diagonal-BQ}: as is easily checked from~\eqref{eqn:paramreln},  the imaginary axis in Figure~\ref{fig:BQ-sets-comparison} maps to the negative real axis in 
Figure~\ref{Diagonal-BQ} while the real axis in Figure~\ref{fig:BQ-sets-comparison} maps to the positive real axis in 
Figure~\ref{Diagonal-BQ}. In particular, note the following branch points and special values:
if $x=3$, then $\z=\pm 1, \pm 3$; if $x=2$, then $\z=\pm\sqrt{3}$; if $x=-1$, then $\z=\pm \sqrt{3}i$; if $x=-2$, then $\z=\pm i, \pm 3i$.

Finally, the symmetry $(x -1/2) \mapsto -(x -1/2) $ is not visible in either picture because it does not preserve the property of lying in the Bowditch set. As we shall see later, this symmetry is visible in pictures of the discreteness locus, see the right hand frame of Figure~\ref{Figs/Riley-Ray-BQ} below.

\section{Discreteness} \label{sec:discrete}

 We now turn to the question of finding those values of the parameter $x$ for which the group 
$\langle X,Y \rangle$ is free and discrete.  
Let $\D_{\S}, \D_{\H} \subset \CC$ denote the subsets of the complex $x$-plane on which the groups $G_{\S}(x), G_{\H}(x)$ respectively are discrete and the corresponding representations are faithful, so that in particular,  $G_{\H}(x)$ is free. (See Section~\ref{sec:whichparameter} for the replacement of $G_{\S}(\xi), G_{\H}(\xi)$ by 
$G_{\S}(x), G_{\H}(x)$.)
 We first  show  that $ \D_{\S}=\D_{\H} $.

We begin with the easy observation that since all the groups in Section~\ref{sec:basicconfig} are commensurable, they are either all discrete or all non-discrete together:
\begin{lemma} \label{lem:commensurable} Suppose that $G, H$ are subgroups of $\PSL$ with $G \supset H$ and that $[G:H]$ is finite. Then $G$  is discrete if and only if  the same is true of $H$.\end{lemma}
\begin{proof} If $G$ is discrete, clearly so is $H$. Suppose that  $H$ is discrete but $G$ is not. Then   infinitely many distinct  orbit points   in $G \cdot O$ accumulate in some compact set $D \subset \HH^3$. Label the cosets of  $[G:H]$ as $g_1H \ldots  g_kH$. Then  for some $i$ there are infinitely many points 
$g_i h_r \cdot O  \in D$, which gives infinitely many distinct points $h_r \in g_i^{-1}D$. This contradicts discreteness of $H$. \end{proof}

\begin{lemma} \label{lem:faithfultogether} The representation  $\rho_{\S}(x) \co \pi_1(\S) \to G_{\S}(x)$ is faithful if and only if the same is true of  $ \rho_{\H}(x) \co \pi_1(\H) \to G_{\H}(x)$. 
\end{lemma}
\begin{proof} 
Note that $\pi_1(\S)$  is isomorphic to  $\ZZ/3\ZZ * \ZZ/3\ZZ = <k_0, k_1 | k_0^3 = k_1^3 = \rm{id}>$, while
$\pi_1(\H)$ is the subgroup of $\pi_1(\S)$ generated by $k_0 k_1$ and $k_1 k_0$
and is isomorphic to a free group of rank $2$. By construction, $\rho_{\H}(x)$ is the restriction of $\rho_{\S}(x)$  to $\pi_1(\H)$.
Thus, if $\rho_{\S}(x)$ is faithful, then so is $\rho_{\H}(x)$.

Now $\pi_1(\H)$ has index three in $\pi_1(\S)$
 and $ \pi_1(\S) = \pi_1(\H) \cup k_0 \pi_1(\H) \cup k_0^{-1} \pi_1(\H)$.
Suppose that $\rho_{\H}(x)$ is faithful but  $\rho_{\S}(x)$ is not.
Then, there exists $g \in \pi_1(S)$ such that $\rho_{\S}(x)(g) = \rm{id}$.
Now $g = k_0^e h$, where $e=\pm  1$ and $h \in \pi_1(\H)$.
Thus $ \mbox{\rm{id} }= \rho_{\S}(x)(g) = \rho_{\S}(x)(k_0^e) \rho_{\H}(x)(h)  $ so that $\rho_{\H}(x)(h^3) = \rho_{\S}(x)(k_0^{-3e}) = \rm{id}$
 contradicting  the assumption that $\rho_{\H}(x)$ is faithful.
\end{proof}

\begin{corollary} \label{discretetogether} The  representations $\rho_{\S}(x), \rho_{\H}(x)$  are faithful and discrete together, that is, $\D_{\S}=\D_{\H} $.
\end{corollary}
Thus we may write $\D = \D_{\S}=\D_{\H} $. Our aim is this section is to find $\D \subset \CC$.

\subsubsection{Fundamental domains}\label{sec:dirichlet}
We can make a rough estimate  for $\D $ by exhibiting  a fundamental domain for $G_{\S}(x)$ for sufficiently large $x$. 

\begin{prop}  \label{prop:dirichlet}Writing  $x = u + iv$, the region $\D$ contains the region  outside the ellipse
$(2u-1)^2/25   +   v^2/4  = 1$ in the $x$-plane.
 \end{prop}
\begin{proof} In view of Corollary~\ref{discretetogether}, we can work with
 the large cone manifold $\S$ with generators $K_0, K_1$ of Section~\ref{sec:largeball}.  By construction, $P$, which conjugates $K_0$ to $K_1$,  is $\pi$-rotation about $0+ j |\z|$ in the upper half space $\HH^3$, which is   the mid-point of the common perpendicular between the axes of $K_0$ and $ K_1$.  The axis of $K_0$ is the line through $ j \sqrt 3$ perpendicular to the real axis.  Let $H, H'$ be the hemispheres which meet $\RR$ orthogonally at points $ -3, 1$ and $-1,3$ respectively, and let $E, E'$ be the closed half spaces they cut out which contain infinity. Then $H, H'$ intersect in $ \Ax  K_0 $, moreover  the complement in $\HH^3$ of $E \cup E'$ is a fundamental domain for the group $\langle K_0\rangle$ acting on $\HH^3$. Likewise the images of $H,H'$ under $P$ meet along $ \Ax  K_1 $ and the complement of 
$P(E ) \cup P(E')$  is a fundamental domain for  $\langle K_1\rangle$. 
 Since $P(z) = \z^2/z, z \in \CC$,  $P(H),P( H')$ meet $\RR$ orthogonally in points $-\z^2/3, \z^2$ and $\z^2/3, -\z^2$ respectively. 
Thus if $|\z| < 1$ then the circle of radius $|\z|$ separates the regions $ E \cap E'$ and $P(E ) \cap P(E')$. We conclude by Poincar\'e's theorem
 (or a suitable simple version of the Klein-Maskit combination theorem) that in this situation $\langle K_0, K_1\rangle$ is discrete with presentation $\langle K_0, K_1| K_0^3 =  K_1^3 = \rm {id}\rangle$. Thus the representation $\rho_{\S}(x)$ with $ x = x(\z)$ as in~\eqref{eqn:paramreln},  is faithful, and hence $x \in \D$.
 
Suppose that $\z = e^{i\phi}$. Then from \eqref{eqn:cxdist}, $(2x-1)/3 =  \cosh \sigma = 1/2( e^{2i \phi}/3 + 3 e^{-2i \phi} )$ so that $x = u+iv$   lies on the ellipse
 $(2u-1)^2/25 + v^2/4 = 1$ as claimed. \end{proof}

\begin{figure}[ht] 
 \includegraphics[width=6cm]{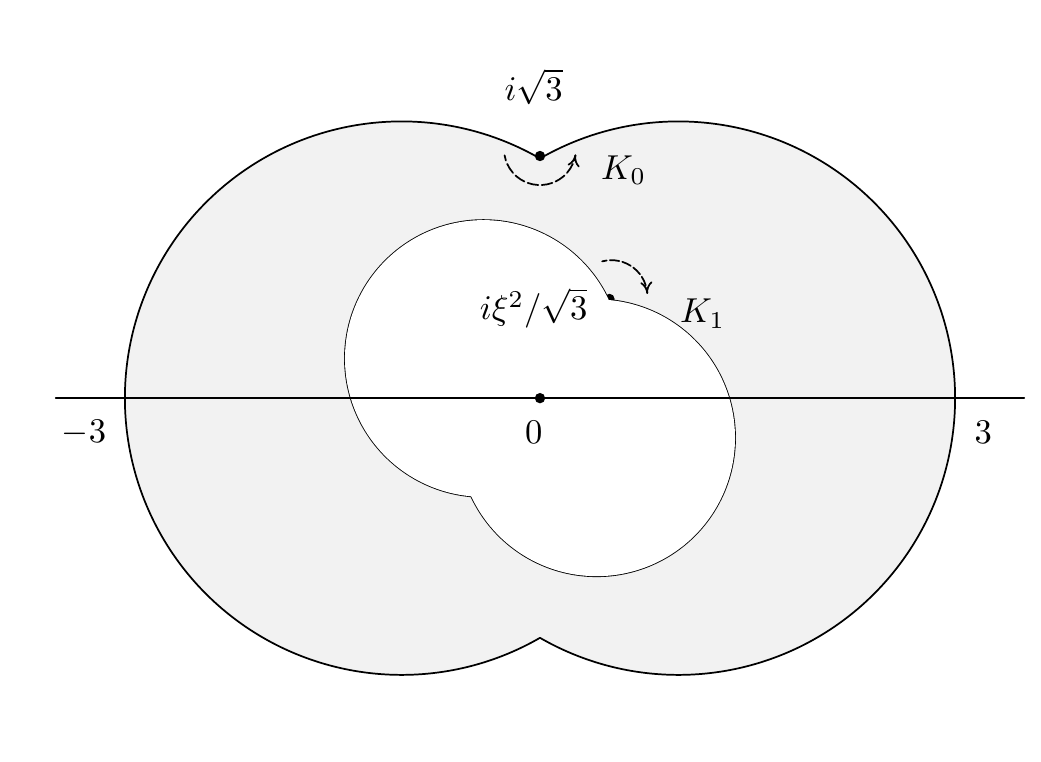}
 \caption{The shaded region illustrates the fundamental domain for $\pi_1(\S)$  acting in its regular set in $\Chat$ when $|\z| >1$, so that $x$ is outside the ellipse of Proposition~\ref{prop:dirichlet}.}\label{fig:dirichletgeneral}
\end{figure}

The configuration when $x \in \RR$ is of particular interest since in this case   $G_{\H}(x)$ is Fuchsian. The ellipse meets the real axis in points $-2,3$ so that $G_{\H}(x)$ is  discrete and the representation is faithful on $(\infty, -2]$ (corresponding to $|\z| >1, \z\in    i\RR$) and 
$[3,\infty)$ (corresponding to $|\z| >1, \z \in    \RR$).  In these two cases the fundamental domains look the same, see Figure~\ref{fig:fuchsianconfigs}.  Note that   the interval $(-2,3)$ is definitely not in $\D$: if  $ -2 < x < 2$ then $K_0K_1   $ is elliptic since
 $x= \Tr K_0K_1   $, while if $ -1 < x < 3$ then  $K_0K_1^{-1}   $ is elliptic since
 $\Tr K_0K_1^{-1}  = 1-x   $, see also Section~\ref{sec:fuchsian}.

In the general case, a fundamental domain can be found by 
a modification of  Wada's   program OPTi ~\cite{OPTi}. 
This program allows one to compute the  limit set and fundamental
 domains for the $PQR$-group $G_{\U}$.  A short python code for doing this is available at \url{http://vivaldi.ics.nara-wu.ac.jp/~yamasita/DiagonalSlice/}

\subsection{The method of pleating rays} \label{sec:pleating}  
To determine $\D$, we  use  the Keen-Series method of pleating rays applied to the large coned sphere $\S$. This is closely analogous to the problem of computing the Riley slice of Schottky space, that is the parameter space of free discrete groups generated by two parabolics, which was solved in~\cite{ksriley,koms}.  

We begin by briefly summarising the elements of pleating ray theory we need.
For more details see various of the first author's papers, for example~\cite{ksriley, chois}.

Suppose that  $G \subset \SL$ is a geometrically finite Kleinian group  with corresponding orbifold  $M = \HH^3/G$  and let $\C/G$ be its convex core, where $\C$ is the convex hull in $\HH^3$ of the limit set of $G$, see~\cite{EpM}. 
 Then  $\dd \C/G$ is a convex pleated surface (see for example~\cite{EpM}) also homeomorphic to $\dd M$. The 
bending of this pleated surface is recorded by means of  a measured geodesic lamination,  the \emph{bending lamination} $\b = \beta(G)$, 
whose support forms the bending lines of the surface   and whose transverse measure records the  total bending angle along short transversals.
 We say  $\beta $ is  \emph{rational} if it is supported on closed curves: note that closed curves in the support of $\beta $ are necessarily simple and pairwise disjoint. If a  bending line is represented by a curve $\gamma \in \pi_1(\S)$, then by definition it is the projection of a geodesic axis to $\dd \C/G$, so in particular $\beta$ contains no peripheral curves in its support. 
Note that any two homotopically distinct non-peripheral simple closed curves on 
$\dd  \S$ intersect. Thus in this case,  $\beta $ is  rational only if its support is a single simple essential non-peripheral  closed curve on $\dd \C/G$.

As above, we parameterise representations $ \rho_{\S}(x) \co \pi_1(\S) \to \SL$ by $x \in \CC$ and denote the image group by $G_{\S}(x)$. From now on, we frequently write $ \rho_{x}$ for $ \rho_{\S}(x)$.

 \begin{definition}
Let $\gamma$ be a homotopy class of simple essential non-peripheral   closed curves on $\dd \S$.
The \emph{pleating ray}  $ \P_{\gamma}$ of   $\gamma$ is the
set of points $ x \in \D$ for which   $ \b(G_{\S}(x)) = \gamma$.  
\end{definition}

Such rays are called \emph{rational pleating rays}; a similar definition can be made for general projective classes of bending lamination, see~\cite{chois}.

The following key lemma is proved in~\cite{chois} Proposition 4.1, see also~\cite{kstop} Lemma 4.6.  The essence is that because the two flat pieces of $\dd \C/G$ on either side of a bending line are invariant under translation along the line, the translation can have no  rotational part.
\begin{lemma}
  \label{lemma:realtrace} If the axis of $g \in G$ is a bending line   of $\dd \C/G_{\S}(x)$, then
$\Tr g(x) \in \RR$.
\end{lemma}
 Notice that the lemma applies even when the bending angle $\th_{\g}$ along $\g$ vanishes, so the corresponding surface is flat, or when the angle is $\pi$, in which case either $\gamma$ is  parabolic or $G_{\S}(x)$ is Fuchsian. 
 
If $g \in G$ represents a curve $\gamma$ on $\dd \S$,  define the \emph{real trace locus} $\RR_{\g}$ of $\g$ to be the locus of points in $\CC$ for which 
$\Tr g \in (-\infty, -2] \cup [2,\infty)$. By the above lemma, $\P_{\g} \subset \RR_{\g}$.

Our aim is to compute  the locus of faithful discrete representations $\D_{\S}= \D$. In summary, we do this as follows:
\begin{enumerate} 
\item Show that up to homotopy in $\S$, the essential non-peripheral curves on $\dd \S$ are indexed by $ \QQ / \sim $  where $p/q \sim \pm (p+ 2kq)/q, k \in \ZZ$. (Proposition~\ref{curveequality}).

\item Given $\gamma \in \pi_1(\dd\S)$, give an algorithm for computing $\Tr \rho_x(\gamma)$  as a polynomial in $x$, in particular identifying its  two highest order terms in terms of $p,q$. (Section~\ref{sec:traces} and Proposition~\ref{prop:tracepoly}).

\item Show that $\P_{0/1} = (-\infty, -3]$ and $\P_{1/1} = [2, \infty)$ (where 
 $\P_{p/q}$ denotes the pleating ray of the curve $\gamma_{p/q} \in \pi_1(\dd\S)$ identified with $p/q$). 
(Section~\ref{sec:fuchsian}).

\item Show that  $\P_{p/q}$ is a union of connected non-singular branches of $\RR_{\g}$.
(Theorem~\ref{thm:nonsing}).

\item For   $p,q \neq 0,1$, identify $\P_{p/q}$ by
showing it has two connected components, namely the branches of $\RR_{\g}$ which are asymptotic to the directions $e^{\pm i \pi (p/q + 1)}$ as $|x| \to \infty$.
(Proposition~\ref{prop:connectivity}).

\item Prove that rational rays $\P_{p/q}$ are dense in $\D_{\S}$.  (Theorem~\ref{thm:density}).
\end{enumerate}

One could carry all this out following almost word for word the arguments in~\cite{ksriley}. Rather than do this, we indicate as appropriate how more general results can be put together to provide a  somewhat less ad hoc proof of the results. The claim that $\P_{p/q}$ has two connected components appears to contradict the results in~\cite{ksriley}, see however  the following remark and Proposition~\ref{prop:connectivity} below.
 The pleating rays are shown on the left in Figure~\ref{Figs/Riley-Ray-BQ}  with the Riley slice rays from~\cite{ksriley} on the right for comparison.

\begin{remark} \label{Komori issue}
\rm{There were two rather subtle errors in~\cite{ksriley} .  The first was  that, in the enumeration of curves on $\dd \S$, we omitted to note that $\gamma_{p/q} $ is homotopic to $ \gamma_{-p/q}$ in $\S$. The second was, that we found only one of the two components of $\P_{p/q}$. 
 Since $\P_{p/q} =   \P_{-p/q}$, these two errors in some sense cancelled each other out. They were discussed at length and resolved in ~\cite{koms} and we make corresponding corrections here.}
\end{remark}

 \begin{figure}[ht] 
\includegraphics[width=5cm]{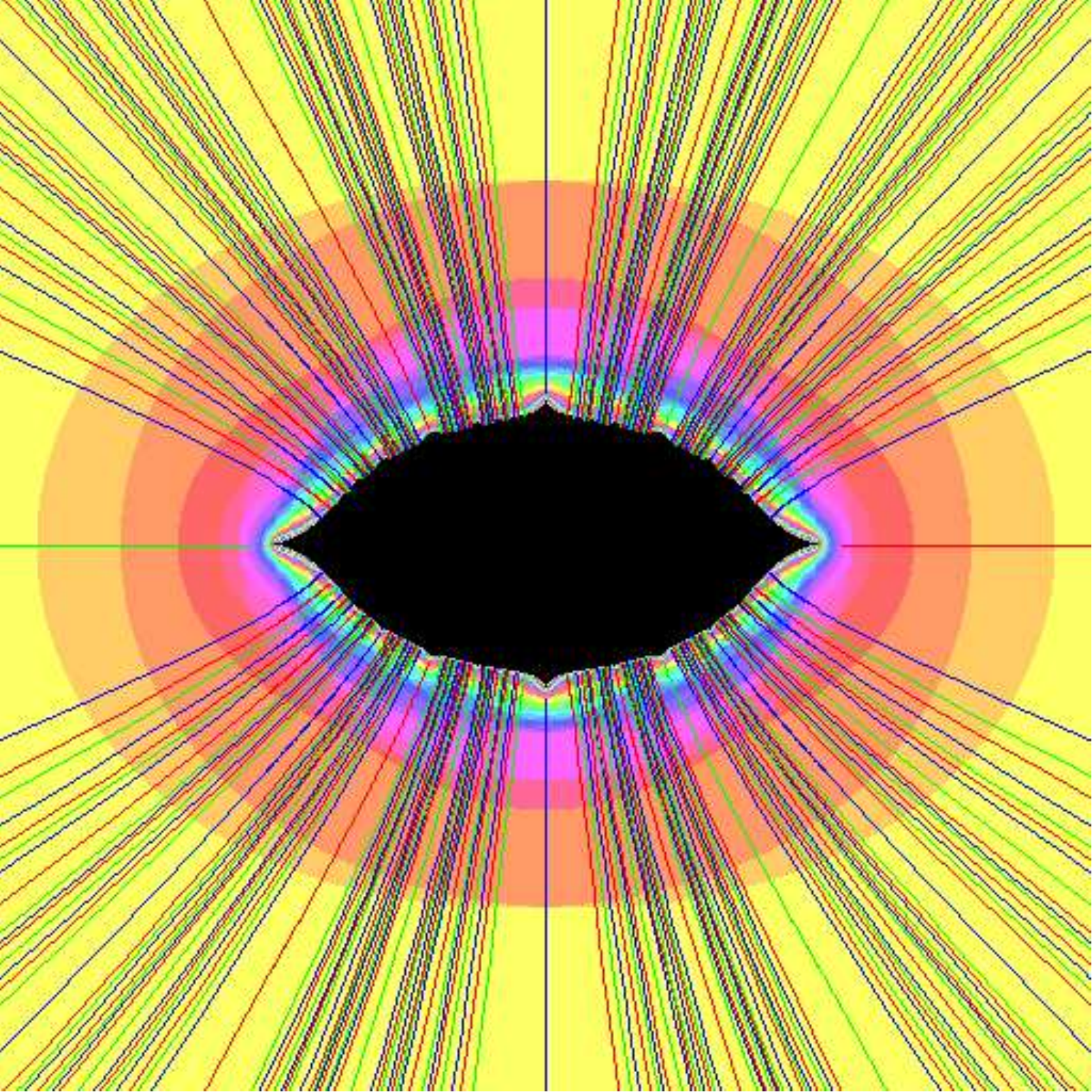}\hspace{1cm}
\includegraphics[width=5cm]{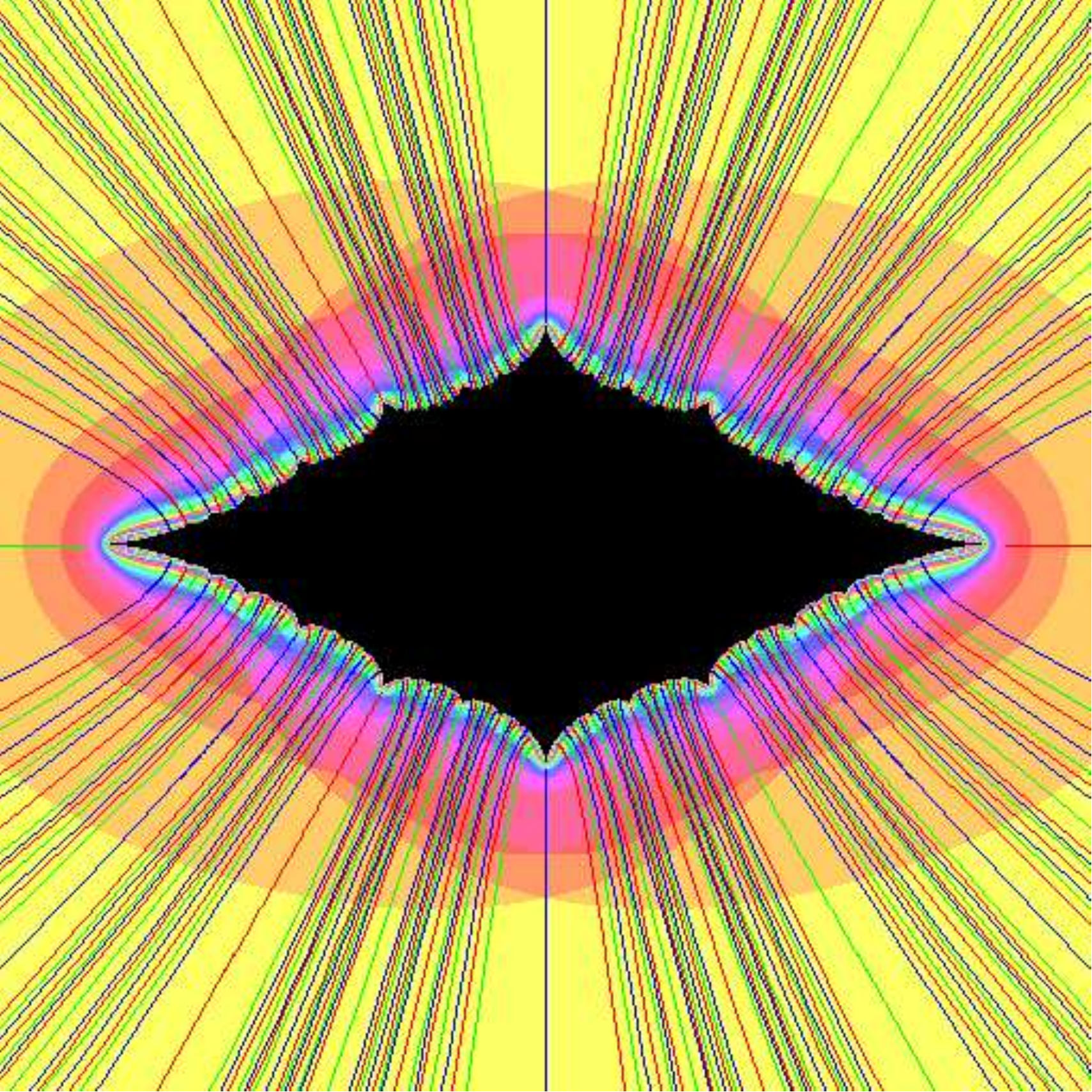}
\caption{Left:  Pleating rays for $G_{\S}(x)$. Right: Pleating rays  for the Riley slice as described in~\cite{ksriley}.  
The coloured (grey) regions are the Bowditch sets for the initial triples discussed in~\ref{sec:discussion}; conjecturally these coincide with the closure of the regions filled by the pleating rays. For a discussion of how the rays were actually computed, see Section~\ref{sec:torustree1}
} \label{Figs/Riley-Ray-BQ}  
\end{figure}

 \subsection{Step 1: Enumeration of curves on $\dd \S$} \label{sec:enumeration}

We need to enumerate essential non-peripheral unoriented simple curves on $\dd \S$ up to homotopy equivalence in $\S$.
As is well known,  such curves on $\dd \S$ are,  up to homotopy equivalence in $\dd \S$,  in bijective correspondence with lines of rational slope
in the plane, that is, with $\QQ \cup \infty$, see for example~\cite{ksriley, koms}. For $(p,q) $ relatively prime and $q \geq0$,  denote the class corresponding to $p/q$ by $\g_{p/q}$.  
We have:
\begin{proposition}[\cite{koms} Theorem 1.2] \label{curveequality}The unoriented curves $\gamma_{p/q}, \gamma_{p'/q'}$ are homotopic in $\S$ if and only if 
$p'/q' = \pm p/q  + 2k, k \in \ZZ$.
\end{proposition}
Missing the identification  $\gamma_{p/q} \sim \gamma_{-p/q}$ was the first of the two errors in~\cite{ksriley} referred to in Remark~\ref{Komori issue}. 

Before proving the proposition, we need to explain the identification of curves on $\dd S$ with $\QQ \cup \infty$. In~\cite{ksriley, koms} this was done using the plane punctured at  integer points as an intermediate covering between $\dd S$ and its universal cover.  The idea is sketched in Section~\ref{sec:correspondence}. Here we give a slightly different description of the curve $\gamma_{p/q} $ which leads to a nice proof of the above result.

 Cut $\S$ into two halves along the meridian disk $m$ which is the plane which perpendicularly bisects the common perpendicular $C$ to the two singular axes $\Ax  K_i , i=0,1$.  Each half  is a ball  $\hat B_i$ with a singular axis $\Ax  K_i $. The boundary  $\dd B_i = \dd \hat B_i \cap \dd \S$  is a sphere with two cone points and a hole $\dd m$.  Since the axes of $K_i$ are oriented, we can distinguish one cone point on each $\dd B_i$ as the positive end  of $\Ax  K_i $. 
Now $\dd \S$ has a hyperbolic structure inherited from the ordinary set (or from the pleated surface structure on $\dd \C/G_{\S}(x)$), in which $\dd m$ is geodesic. With respect to such a structure, 
each $\dd B_i$ has a reflectional symmetry $\iota$ in the plane containing $\Ax  K_i $ and $C$, which maps the cone points to themselves, is an involution on  $\dd m$ with two fixed points which maps the `front' to the `back' as shown in Figure~\ref{fig:fronttoback}.    There is a preferred base point $P_i$ on $\dd m$, namely the foot of the perpendicular from the negative end of  $\Ax  K_i $ to $\dd m$.

 \begin{figure}[ht] 
\includegraphics[width=8cm]{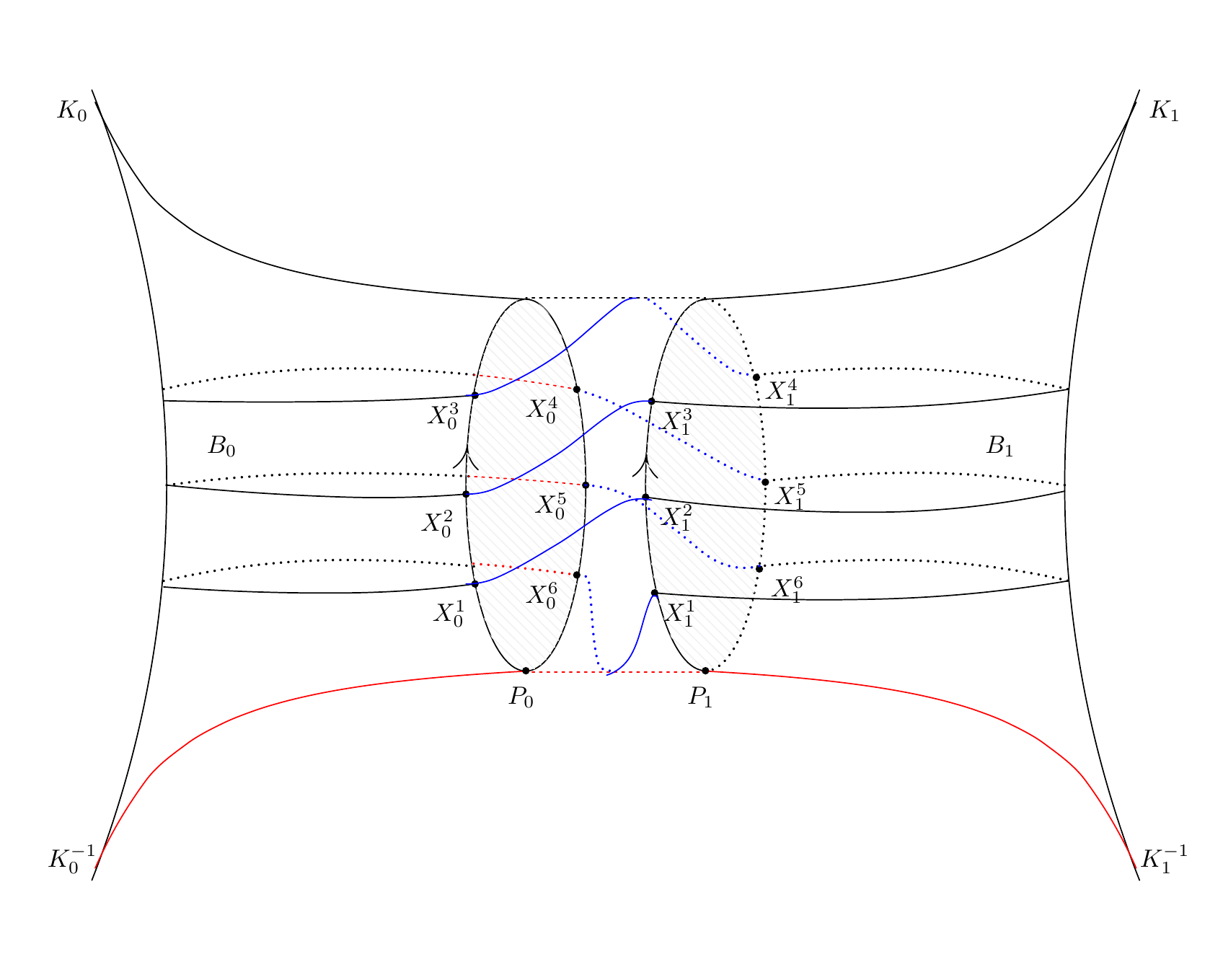}
\caption{The arrangement of arcs on $\dd \S$. The curve shown illustrates the case $p=1, q=3$.}\label{fig:fronttoback}
\end{figure}

Let $\g$ be an essential non-peripheral simple curve on $\dd \S$. For each $i= 0,1$,  $\g \cap \dd B_i$ consists of $q$ arcs joining $\dd m$ to itself.  Start with the strands of $\g$ arranged symmetrically with respect to $\iota$, that is, with front to back symmetry.
 Orient $\dd m$ so that it points `upwards' on the front side of the figure, noticing that $\iota$ reverses the orientation. Lifting $\dd m $ to its cyclic cover $\RR$, enumerate in order the endpoints $X^{k}_i, i=0,1;  k \in \ZZ$ of arcs of $\gamma$   starting (say) with the arc meeting 
$\dd m$ nearest $P_i$, and so that increasing order is in the direction of the upwards orientation of $\dd m$ viewed from the front side in the figure. Since in fact $X^{k}_i = X^{k+2q}_i $, the enumeration is really $\rm{mod} \  2q\ZZ$.   
 
   To reconstruct  $\g$ we have to join the endpoints $X^{k}_0$ on $\dd B_0$ to the endpoints $X^{k'}_1$ on $\dd B_1$.
  Since the arcs have to be matched in order round $\dd m$, if  $X^i_0$ is  joined to $X^j_1$ then 
$X^{i+k}_0$ is  joined to $X^{j+k}_1$ for all $k \in \ZZ$. Set  $ p = j-i$.

It is not hard to see that the resulting curve $\gamma_{p/q}$ is connected if and only if $(p,q)$ are relatively prime. Note that with this description, $\dd m$ is the curve $q=0$, that is $\g_{1/0}$. The curve $\g_{0/1}$ is the curve $K_0K_1$ and $\g_{1/1} = K_0K_1^{-1}$.
We leave it to the reader to see that this description is the same as that obtained from the lattice picture in~\cite{ksriley}, see also Section~\ref{sec:correspondence}.

\medskip
{\sc Proof of Proposition~\ref{curveequality}}
Write $\gamma_{p/q} \sim \gamma_{p'/q'}$ to indicate that $\gamma_{p/q} , \gamma_{p'/q'}$ are homotopic in $\S$.
Since Dehn twisting round $\dd m$ is trivial in $\S$ and sends $X^k_i \to X^{k+2q}_{i}$,  we have $\gamma_{p/q} \sim \gamma_{p/q +2}$. To see why $\gamma_{p/q} \sim \gamma_{-p/q}$, we proceed as follows. Consider the arrangement  shown in Figure~\ref{fig:fronttoback}, in which $B_0$ is on the left  and  $\dd m$ is oriented `upwards'. Each $\dd  B_i$ has  a natural `front' and `back' which are interchanged by the involution $\iota$.  
We will prove the result by arranging $\gamma$ on $\dd S$ in two different ways. Start with the strands of $\g$ arranged symmetrically front to back, that is, with respect to $\iota$, as on the two sides of Figure~\ref{fig:fronttoback}.

(1) Homotope  all the arcs of $\gamma \cap \dd B_i$ on $\dd B_i$  by dragging the endpoints on the back side around one or other cusp so that they meet $\dd m$ on the front, see Figure~\ref{fig:pullround}. In this case, the choice $p>0$ means that 
on an arc approaching $\dd m$ from $\dd B_0$ one turns left (up) for $p$ slots before joining up arcs.

 \begin{figure}[ht] 
  \includegraphics[width=4cm]{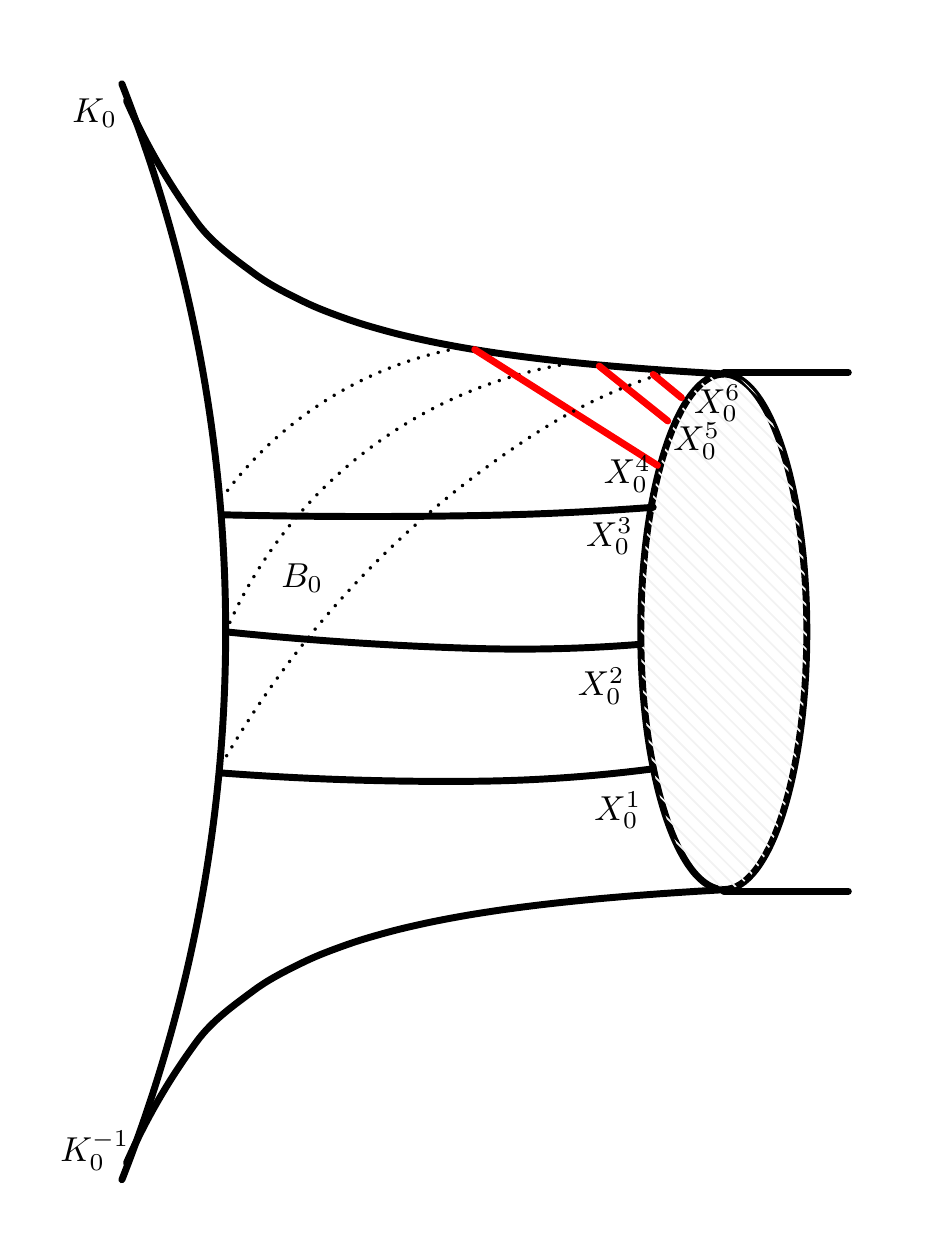}
 \caption{Curves pulled round to the front of $B_0$ ready to be joined as in (1) of the proof of Proposition~\ref{curveequality}.}\label{fig:pullround}
\end{figure}

(2) Now for the second arrangement. Homotope  all the arcs of $\gamma \cap \dd B_i$ on $\dd B_i$, by dragging the endpoints on the front side  so that they meet $\dd m$ on the back of $\dd B_i$. Notice turning left for $p$ slots on the back side, when viewed through $\S$ from the front to the back, looks the same as turning right (down) for $p$ slots on the front side.

Now, starting from situation (2),  homotope $\gamma$ in $\S$ by pulling the arcs meeting $\dd m$ on the back side through $m$ so that they meet $\dd m$ on the front side keeping their `horizontal' level fixed,  so that an endpoint $X$ moves to $\iota(X)$. The endpoints on the front side now have  the same up-down order as they did on the back side. This shows that  the curve obtained connecting the arcs moving up $p$ slots is homotopic in $\S$ to the curve obtained by moving down $p$ slots, as claimed.  

We will show that  if $p'/q' \neq \pm p/q  + 2k, k \in \ZZ$ then $\gamma_{p/q} \nsim \gamma_{p'/q'}$ after computing traces, see Corollary~\ref{curveinequality}. \qed

\subsection{Step 2: Computation of traces}\label{sec:traces}

Let $V_{p/q} \in \SL = \rho_x(\g_{p/q})$, where, since we want to compute $\Tr V_{p/q}$,  we only need consider $V_{p/q}$ up to cyclic permutation and inversion.  
Rather than using the associated torus tree, we will work directly with a $4$-holed sphere $\Sigma_{0,4}$ and the associated tree as described in~\cite{MPT}, see also \cite{goldman}. 
Let $\alpha, \beta, \gamma,\delta$ denote loops round the four holes, oriented so that $\alpha  \beta \gamma \delta = \rm{id}.$ The fundamental group is identified with the free group $  F_3$ with generators
$\alpha, \beta, \gamma$.
 A representation $\rho\co F_3 \to \SL $ is determined up to conjugation by  its values on seven elements as follows (where we use $\hat w$ in place of $w$ in \cite{MPT} etc to distinguish it from a variable $w$ already in other use):
$$ \Tr \rho(\a) = a; \Tr \rho(\b) = b; \Tr \rho(\c) = c; \Tr \rho(\d) = d$$
$$ \Tr \rho(\a\b) = \hat x; \Tr \rho(\b\c) = \hat y; \Tr \rho(\c\a) = \hat z$$
related by the equation 
\begin{equation} \label{conedspheretraces}  \hat x^2 +\hat y^2 +\hat z^2 +\hat x\hat y\hat z = \hat p\hat x + \hat q\hat y+\hat  r\hat z + \hat  s\end{equation} 
where  
$$ \hat p = ab + cd,  \hat  q = bc+ad,  \hat  r = ac+bd,  \hat  s = 4-a^2 -b^2 - c^2 - d^2 - abcd.$$
We identify our generators $K_i$ as:
$ \alpha = K_0, \b = K_1, \c = K_2, \d = K_3$.
Thus we find:
$$ a = b = c = d = 1, \hat x = \Tr K_0K_1 = x, \hat y =  \Tr K_1K_2 = 2, \hat z =  \Tr K_2K_0 = -x+1.$$
As a check, it is easy to verify that the trace identity \eqref{conedspheretraces} holds. Notice that none of the expressions  $\hat p, \ldots, \hat z$ depend on the sign choices made in Section~\ref{sec:handlebody}.

The traces can be arranged in a trivalent tree in the usual way. As explained above, we have $\gamma_{0/1} = K_0K_1, \gamma_{1/0} =   \rm{id}, \gamma_{1/1} =K_0K_1^{-1} $.  As explained in~\cite{MPT} Section 2.10, there are now $3$ moves, depending on the values of $\hat p,\hat q, \hat r$. In our case $\hat p=\hat q= \hat r = 2$ so the three moves described there coincide. 
Following~\cite{MPT}, if $u,v,w$ are labels round a vertex, with $v,w$ labels adjacent along a common edge $e$, then the label at the vertex at the opposite end of $e$ is $u' = 2-vw-u$, compare Figure~\ref{fig:markofftree} in which $u ' = vw-u$.

Clearly this procedure gives an algorithm for arranging curves and computing  traces on a trivalent tree by analogy with that described in Section~\ref{sec:markoff}. 
Curves generated  in this way inherit  a natural labelling from  the usual procedure of Farey addition as described in Section~\ref{sec:markoff}. Denote the curve  which inherits the label $p/q$   by $\delta_{p/q}$; we say this curve is in \emph{Farey position $p/q$} on the tree. We shall refer to this tree together with its new rule for computing traces  as the $\S$-tree, to distinguish it from the Markoff tree of Section~\ref{sec:markoff}. 

 We need to show that $\delta_{p/q}$ is the same as the curve  $\gamma_{p/q}$ described in the previous section, namely  the 
curve with $2q$ intersections with the meridian $\dd m$ and a twist by $p$.  
 \begin{lemma} \label{lem:treeisok}  With the above notation,  $\d_{p/q} = \g_{p/q}$.\end{lemma}
\begin{proof} (Sketch) By definition we have $\d_{p/q} = \g_{p/q}$ for $p,q \in \{0,  1 \}$. With the notation above,  these are the curves $\a \b, \b \g, \g \a$, each of which separates the punctures in pairs.  

Call two essential simple non-peripheral curves on $\dd \S$ \emph{neighbours} if they intersect exactly twice. 
Note that of the initial triple,  each pair adjacent along an initial edge are neighbours, so that the triple round the initial vertex are neighbours in pairs.

Now we check inductively that this is always the case.  Give a pair of neighbours $\d, \d'$ along an edge, the remaining curves at the vertices at the opposite ends of this edge are obtained by Dehn twisting $\d$ about $\d'$ (or vice versa) in each of the two possible directions, see~\cite{MPT}. Thus for example $\d_{1/1}$ is obtained by  Dehn twisting $\d_{0/1}$ about $\d_{1/0}$ while $\d_{-1/1}$ is obtained by  Dehn twisting $\d_{0/1}$ about $\d_{1/0}$ in the opposite direction. Moreover if $\d, \d'$ are neighbours, then so are the pairs
$\d, D^{\pm}_{\d}( \d')$ and $\d', D^{\pm}_{\d}( \d')$, where $D_{\d}( \d')$ denotes a Dehn twist of $\d'$ about $\d$ and by abuse of notation we write $D^{+}, D^{-}$ for $D, D^{-1}$ respectively.

Now we show inductively that $\d_{p/q} = \g_{p/q}$. Suppose that it is true for neighbours $p/q,r/s$ where $|ps-rq| = 1$. By induction we may assume that $\d_{p/q} , \d_{r/s}$ are adjacent along an edge $e$ of the tree. The two remaining curves $D^{\pm}_{\d_{p/q}}(\d_{r/s})$ at the vertices at the ends of $e$   result from Dehn twisting $\d_{p/q}$ about $\d_{r/s}$ in each of the two possible directions, hence are also neighbours. From the inductive hypothesis, we can assume that one of these two curves is $\d_{p-r/q-s} = \g_{p-r/q-s}$ or $\d_{r-p/s-q} = \g_{r-p/s-q}$  depending on whether  $ ps-rq  = \pm 1$. In accordance with the Farey labelling system, the curve at the other vertex is $\d_{p+r/q+s} $.

 \begin{figure}[ht] 
  \includegraphics[width=8cm]{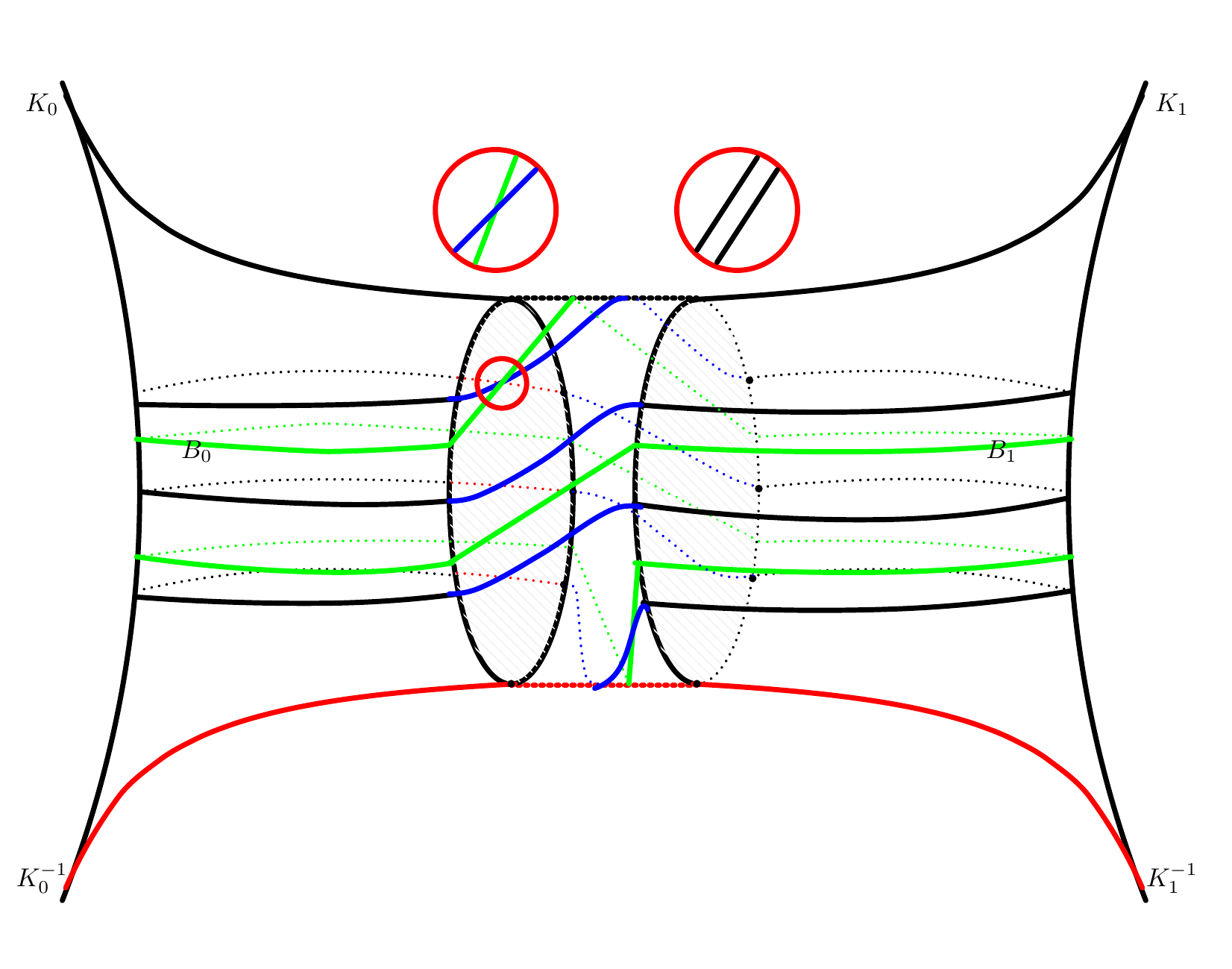}
 \caption{Here  $\g_{1/3}$ and $\g_{1/2}$ are surgered to give $\g_{2/5}$, see the proof of Lemma~\ref{lem:treeisok}. The inset circles show the direction of surgery. }\label{fig:surgery}
\end{figure}

Now  the curves
 $D^{\pm}_{\d_{p/q}}(\d_{r/s})$ can be found by surgery. On each $B_i$, arrange both curves symmetrically with respect to the front and back of $\S$ as described above, then join the strands so that they have minimal intersection as in Figure~\ref{fig:surgery}. With $\g_{p/q}$ in this position, its twist $p$  is its intersection number with the geodesic joining the two positive cone points of the axes $K_i$,  and likewise for 
$\g_{r/s}$.  To perform the Dehn twist we have to cut the curves at their intersection points and then make a consistent choice of which direction to rejoin the resulting arcs.  One of the two choices will give a curve with $2(q+s)$ intersection  points with the meridian $\dd m$. Clearly the curve with the `positive' surgery (see the inset circles in Figure~\ref{fig:surgery}) will have intersection number $p+r$ with this line, and hence is the curve $\g_{p+r/q+s}$. Since we already know the curve at the other vertex is  $\d_{r-p/s-q} =\g_{p-r/q-s}$ (or $\g_{r-p/s-q}$), this shows that $\d_{p+r/q+s} = \g_{p+r/q+s}$. This completes the proof.
\end{proof}

\begin{prop} \label{prop:tracepoly} Let $V_{p/q}(x) = \rho_{\S}(x) (\g_{p/q})$ as above. Then:
\begin{enumerate}
\item $\Tr V_{p/q}$ is a  polynomial  in $x$ whose top two terms are $(-1)^{p-q-1} (x^{q} - px^{q-1} )$.

\item $\Tr V_{p/q} = \Tr V_{(p/q) + 2} = \Tr V_{-p/q}$.

\item $\Tr V_{p/q} (x) = \Tr V_{p+q/q } (1-x)$.

\end{enumerate}

\end{prop}

\begin{remark}\rm{(1) should be compared to~\cite{ksriley} Corollary 4.3 in which we showed   that the leading term is of the form $(-1)^{p-q-1}c x^q$ for some $c>0$, see also the remark following the corollary in that paper. }\end{remark}
\begin{proof}  (1) 
Note that (1) holds for the three initial traces of $  \gamma_{0/1},  \gamma_{1/0}
,  \gamma_{1/1} $. If curves $  \gamma_{p/q},  \gamma_{r/s}$ are adjacent along an edge, then the two curves at the remaining vertices at the ends of the edge are $  \gamma_{p\pm r/q \pm s}$.
The result then follows easily by induction on the tree.

(2) This follows immediately from Proposition~\ref{curveequality} and can also be proved easily  by looking at the symmetries of the $\S$-tree.

(3) This results from the symmetry $ x \mapsto 1-x$ which interchanges $  \gamma_{0/1},  \gamma_{1/1}$.

\end{proof}

Now we can prove the `only if' assertion of Proposition~\ref{curveequality}:
\begin{corollary} \label{curveinequality}  If $p'/q' \neq \pm p/q  + 2k, k \in \ZZ$ then $\gamma_{p/q} \not\sim \gamma_{p'/q'}$.

\end{corollary}
\begin{proof} This follows immediately by comparing the top two terms of $\Tr W_{p/q}, \Tr W_{p'/q'}$.
\end{proof}

\subsection{Step 3. The exceptional Fuchsian case: computation of $\P_{0/1}, \P_{1/1}$} \label{sec:fuchsian}

 As above, let $\P_{p/q}$ denote the pleating ray of $\gamma_{p/q}$.   
The rays $\P_{0/1}, \P_{1/1}$ are exceptional. We have $\Tr \gamma_{0/1}= x, \Tr \gamma_{1/1} = 1-x$. Thus
the real locus for both trace polynomials is exactly the real axis, and on this locus, the group $G_{\S}(x)$, if discrete, is Fuchsian.
This is exactly the situation discussed in~\cite{ksriley} p. 84.

In the ball model of $\HH^3$,  identify the extended real axis with the equatorial circle. Since the limit set is contained in $\hat \RR$, the convex core (the Nielsen region) of $G_{\S}(x)$
 is contained in the equatorial plane. We can think that the convex core has been squashed flat and the bending lines are just the boundary of the Nielsen region, that is, the boundary of the surface $\HH^2/G_{\S}(x)$. Thus to find the bending lamination we just have to determine  the boundary of $\HH^2/G_{\S}(x)$. 
 
 Now if  $x \in \RR$ then either $\z \in \RR$ and  $x>0$, or $\z \in i\RR$ and $x<0$. 
 In both cases, we find a fundamental domain for $G_{\S}(x)$ as  described in~\ref{sec:dirichlet}, see Figure~\ref{fig:fuchsianconfigs}. Thus regarded as a Fuchsian group acting on the upper half plane $\HH$, $G_{\S}(x)$ represents a sphere with two order $3$ cone points and one hole.  However the cases $ x<0, x>0$ are slightly different, because of the relative directions of rotation of $K_0$ and $K_1$. 
 
 In both cases, the axis $K_0$ has fixed points $\pm i \sqrt 3$ and its axis is oriented so that it is anticlockwise rotation about $  i \sqrt 3$. Thus $K_1 = PK_0P^{-1}$ rotates anticlockwise  about 
 $P(  i \sqrt 3) = -i \z ^2/\sqrt 3 $. If $x<0$ then  $P(  i \sqrt 3)$ is in the upper half plane $\HH$ while  if $x>0$ then  $P(  i \sqrt 3)$ is in the lower  half plane. Hence if $x<0$ then $K_0, K_1$ rotate in the same sense about their fixed points in $\HH$ while if $x>0$ their rotation  directions are opposite. This leads to the two different configurations shown in Figure~\ref{fig:fuchsianconfigs}.
 
 As is easily checked, if $x>0$ the boundary of the hole is thus $K_0K_1^{-1}$ while if $x<0$ the boundary of the hole is $K_0K_1$.  Since $K_0K_1 = \gamma_{0/1}$ and $K_0K_1^{-1} = \gamma_{1/1}$, combining this with information about the discreteness locus in the Fuchsian case from~\ref{sec:dirichlet},  we conclude that $\P_{0/1} = (-\infty, -2]$ and $ \P_{1/1}= [3, \infty)$.

 \begin{figure}[ht] 
 \includegraphics[width=6cm]{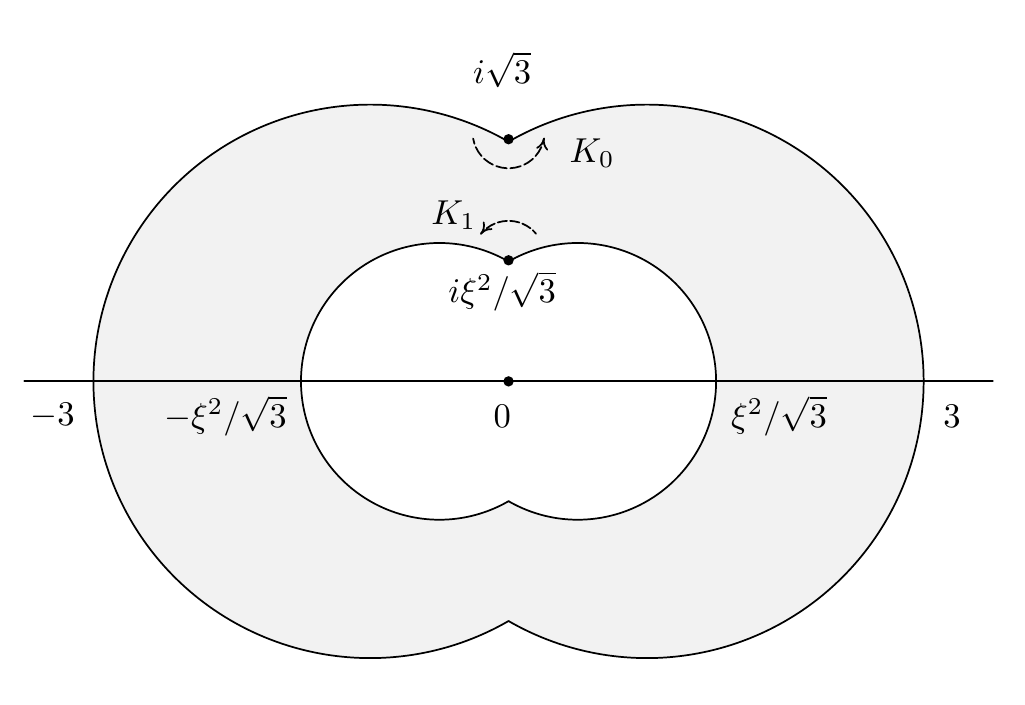}
  \includegraphics[width=6cm]{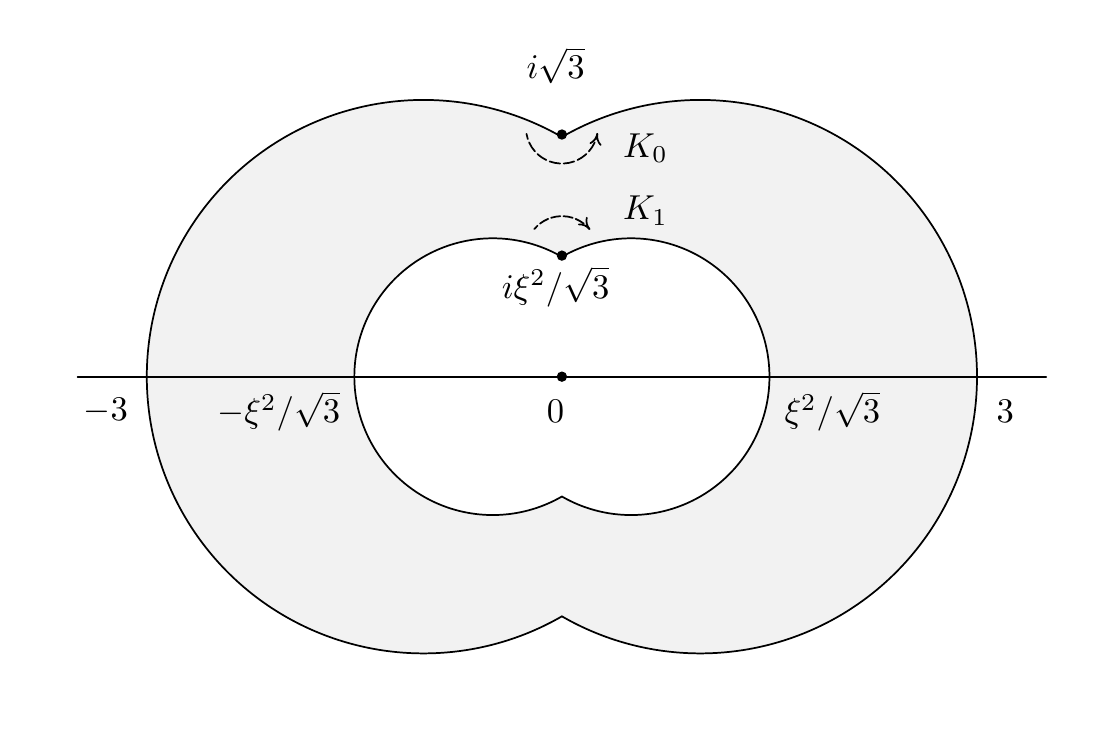}
 \caption{Configurations  for $x \in \RR$. Left: $\z \in i\RR, x\leq -2$. $K_0$ and $ K_1$ rotate in the same  directions   $\HH$ and the hole is $K_0K_1$.
Right: $\z \in  \RR, x\geq 3$.  $K_0$ and $ K_1$ rotate in opposite directions in $\HH$ and the hole is $K_0K_1^{-1}$.}\label{fig:fuchsianconfigs}
\end{figure}

\subsection {Step 4. Non-singularity of pleating rays}\label{sec:nonsingular}
   
 This is the part of the argument which contains the deepest mathematics. Fortunately the results needed have been proved elsewhere.

 \begin{theorem}[\cite{ksqf, ksriley, chois}] \label{thm:nonsing}
 Suppose that $\gamma \in \pi_1(\S)$. Then $\P_{\gamma}$ is open and closed in the real trace locus $\RR_{\g}$. 
 Moreover $\Tr \rho_x(\g)$ is a local coordinate for $\CC$  in a neighbourhood of $\P_{\gamma}$, and  is a global coordinate for $\P_{\gamma}$ on any non-empty connected component of $\P_{\gamma}$. 
\end{theorem}
\begin{proof} The  statement  that   $\P_{\gamma}$ is open in $\RR_{\g}$ is essentially \cite{ksriley} Proposition 3.1, see also   \cite{ksqf} Theorems 15  and 26.  The fact that $\Tr \rho_x(\g)$ is a local parameter is equivalent to the fact, also proved in both \cite{ksriley}  and   \cite{kstop}, that $\P_{\gamma}$ is a non-singular $1$-manifold. The open-ness and the final statement are actually  a 
special case of  Theorems B and C  of \cite{chois} which state that for general hyperbolic manifolds, if the support of the bending lamination is a union of closed curves  is rational, then the traces of  these  curves are local parameters for the deformation space in a neighbourhood of  the corresponding pleating variety.

That   $\P_{\gamma}$ is closed in $\RR_{\g}$ can be  proved as in~\cite{ksriley} Theorem 3.7. Here is a slightly more sophisticated version of the same idea. Suppose $x_n \to x_{\infty}$ with $x_n \in \P_{\gamma}$. The limit group 
$G_{\S}(x_{\infty})$ is an algebraic limit of groups $G_{\S}(x_n)$ and hence the corresponding representation is discrete and faithful. Each of the two components of $(\dd \C/G_{\S}(x_n)) \setminus \g $ is a flat surface corresponding to a conjugacy class of Fuchsian subgroup $F_j(x_n), j=1,2$ (the $F$-peripheral subgroups of~\cite{ksriley}).  Since the limit is algebraic, $F_j(x_n)$ limits on a Fuchsian subgroup $F_j(x_{\infty})$, and similarly for all its conjugates in $G_{\S}(x_{\infty})$. 

The  limit sets $\Lambda_{\alpha}$ of each of these subgroups $F_{\alpha}$ is spanned by a hyperbolic  plane $H_{\a}$ in $\HH^3$. The Nielsen regions  of  $F_{\alpha}$ in $H_{\a}$ fit together along the
 lifts of the bending line $\g$ to $\HH^3$, forming a pleated surface $\Pi$ in $\HH^3$. We claim that $\Pi =  \dd \C (x_{\infty})$. This follows since the closure of the union of the $\Lambda_{\a}$ is the limit set of $G(x_{\infty})$, see also
 Proposition 7.2 in~\cite{kshowtobend}. The  result follows.   \end{proof}

\begin{remark} \rm{The closure of $\P_{\g}$ in $\RR_{\g}$ is a simple case of both the `local limit theorem', Theorem 15 in~\cite{ksqf} and the  `lemme de fermeture' of~\cite{BonO}. These  much more sophisticated results allow that the bending lines may be part of an irrational lamination. Our argument above, in which the bending lamination is supported on closed curves,  is very close to that in the first part of the proof of Theor{\`e}me 6 in~\cite{otal}. }   \end{remark}

\begin{corollary}[\cite{kstop, ksriley, chois}]  If $\P_{\gamma} \neq \emptyset$, then it is a union of  connected non-singular branches of the real trace locus $\RR_{ \gamma}$. 
\end{corollary}
\begin{proof}  Suppose that $\P_{\gamma} \neq \emptyset$ and let $x \in \P_{\gamma}$, so that by Lemma~\ref{lemma:realtrace}, $x \in \RR_{\g}$.  By Theorem~\ref{thm:nonsing},  $\P_{\g}$ is open and closed in $\RR_{\g}$. Since $\Tr \gamma$ is a local coordinate, in a neighbourhood of $x$ the locus $\RR_{\g}$
is a $1$-manifold.  
\end{proof} 

Notice that the theorem says that $\Tr \rho_x(\g)$ is a local parameter even in a neighbourhood of a cusp where $ \rho_x(\g)$ is parabolic,~\cite{ chois} Theorem C.
Thus we have
\begin{corollary} Suppose that $ x \in \P_{\gamma}$. Then there is a neighbourhood of $x$ in $\CC$ on which $x \in \RR_{\g}$ implies that $ x \in \D$.
\end{corollary}

 \begin{corollary}\label{cor:unbounded} If $\P_{\gamma} \neq \emptyset$, then $\Tr \rho_x(\g)$ is unbounded on $\P_{\gamma}$. 
\end{corollary}
\begin{proof}
Since $\Tr \rho_x(\g)$  is a local coordinate on connected components of $\P_{\gamma}$, this follows from the maximum principle on the branch, see~\cite{ksriley} Theorem 4.1. 
\end{proof}

\subsection{Step 5. Finding the non-empty pleating rays} \label{sec:general rays}

Now we determine the pleating rays. As above, let $\P_{p/q}$ denote the ray corresponding to 
the curve $\g_{p/q}$ and write $\RR_{p/q}$ for the real locus of $\Tr V_{p/q}$. From Proposition~\ref{curveequality} we have $\P_{p/q} = \P_{(p+2q)/q}= \P_{-p/q}$.

By~\ref{sec:nonsingular}, $\P_{p/q}  $ is a union of non-singular branches of $\RR_{p/q}$.
We  now find those $p/q \neq \{0,1\}$ for which $\P_{p/q} \neq \emptyset$,  at the same time resolving the connectivity issue.   We follow the method of \cite{ksriley}, using  an inductive argument on position of the pleating rays and their asymptotic directions as $|x| \to \infty$, and at the same time correcting the second of the two errors referred to in Remark~\ref{Komori issue}.   We have:
 \begin{proposition}[c.f.~\cite{ksriley} Theorem 4.1]\label{prop:asympdirn}
The set $\P_{p/q}$ is the union of the two branches  of $\RR_{p/q} $ which are asymptotic to the half lines $\rho e^{\pm i \pi (p-q)/q}$ as $\rho \to \infty$.
 \end{proposition}
\begin{proof} 
Denote by $R(\theta)$ the ray  $te^{i \theta}, t >0$, in the $x$-plane. By Proposition~\ref{prop:tracepoly}, $\Tr V_{p/q}$ is a polynomial in $x$ whose top term is $(-1)^{p-q-1} x^q$. Now $\Tr V_{p/q}$ takes real values on $\P_{p/q}$,  moreover by Corollary~\ref{cor:unbounded} it is unbounded on $\P_{p/q}$. It follows that 
as $|x| \to \infty$, $\P_{p/q}$ must be asymptotic to one of the rays $R(k\pi/q), k \in \ZZ$.

We have already identified $\P_{0/1}$ and $\P_{1/1}$ as the real intervals $(-\infty, -3]$ and $[2,\infty)$   respectively. It follows from Section~\ref{sec:dirichlet} that the  semicircular arc from $-4$ to $4$ (say) in $\HH$ is a continuous path in $\D$ from 
$\P_{0/1}$ to $\P_{1/1}$. Hence by the continuity theorem of~\cite{kscont},  if  $0 < p/q <1$ there is a point on  $\P_{p/q}$ in the upper half plane $\HH$. Likewise there is a point 
on  $\P_{p/q}$ in the lower half plane. (This was missed in~\cite{ksriley}.) Since $\P_{0/1} \cup \P_{1/1}$ separates $\D$ into two connected components, this shows in particular that  $\P_{p/q}$ must have at least two connected components. 

 Now we proceed by induction on the Farey tree. Suppose we have shown the result for two Farey neighbours $p/q, r/s$. Consider the locus $\P_{p+r/q+s}$. By the inductive hypothesis,  $\HH$ contains exactly one component  of  each of  $\P_{p/q}, \P_{r/s}$, asymptotic to the rays  $R(\pi (p-q)/q), R( \pi (r-s)/s)$ respectively. Exactly as in~\cite{ksriley}  it is easy to check that there is exactly one integer $k \in \{0, 1, \ldots, 2(q+s)-1 \}$ for which $R(k\pi/(q+s))$ lies between $R(\pi (p-q)/q)$ and $ R( \pi (r-s)/s)$, namely $k= (p+r)/(q+s)$. By the same continuity theorem as before, a path in this sector joining suitable points on $\P_{p/q}, \P_{r/s}$ must meet $\P_{p+r/q+s}$. Thus $\P_{p+r/q+s}$ has at least one connected component asymptotic to $ R( \pi (p + r -q-s)/(q+s))$.
 A similar argument in the lower half plane gives another connected component asymptotic to $ R( \pi (p + r +q+s)/(q+s))$. Since $\P_{p+r/q+s}$ has exactly two components by Proposition~\ref{prop:connectivity} below, the result follows.
\end{proof}

The issue of connectivity of $\P_{\g}$ is a bit subtle. In the general theory, see~\cite{BonO, chois}, one shows that $\P_{\g}$ has one connected component.   However this result holds in a space of manifolds which are consistently oriented throughout the space and
all of whose  convex cores have non-zero volume. In our case we have:

\begin{proposition} \label{prop:connectivity} If $\g  \neq 0/1, 1/1$ and $\P_{\g} \neq \emptyset$,  then  $\P_{\g}$ has exactly two connected components in $\D$.
\end{proposition}
\begin{proof}
The usual argument that the pleating ray of a rational lamination  has one connected component goes as follows.  Given a point on $\P_{\g}$, double the convex core along its boundary to obtain a cone manifold with a singular axis of angle $2(\pi - \theta)$ along $\g$, where $\theta$ is the bending angle along $\gamma$. (Notice that the convention on defining bending angles differs between papers by the first author and~\cite{BonO}. In our convention, a bending line contained in  flat subsurface has bending angle $0$ but cone angle $2 \pi$, whereas in~\cite{BonO},  the bending angle along a line in a flat  surface is defined to be $\pi$.) By \cite{HK},  such a  hyperbolic cone manifold is  parametrized by its cone angle. One shows that one can continuously deform the cone angle to $0$, at which point the curve whose axis is the bending line has to become parabolic.  The doubled manifold is an oriented hyperbolic manifold with a rank two cusp and finite volume. As long as we are working in a space in which all manifolds have consistent orientation, such a manifold  is  unique up to orientation preserving isometry, from which one deduces that $\P_{\gamma}$ is connected.

In our case, the parameter space $\D$ is separated by two lines along which $G$ is Fuchsian so that $\C(G)/G$ has zero volume and  the above argument fails.
Note however that, provided that $G$ is not Fuchsian,   $\S$ can be oriented by the triple consisting of the \emph{oriented} axes of $P,Q$ and the oriented line $C$ from $ \Ax  K_1 $ to $ \Ax  K_1 $. 
The map $ \z \to \bar \z$ reverses the relative orientations of $ \Ax P,  \Ax Q$ while fixing that of $C$.   Thus $\D \setminus \RR$ has two connected components in which $\S$ has naturally opposite orientations. The above argument shows that  $\P_{\gamma} $ has at 
most  one component  in each component of $\D$.  
Since we have already shown in Proposition~\ref{prop:asympdirn}
 that $\P_{\gamma} $ has at least one component in each of the upper and lower half planes, this completes the proof. 

This proposition can alternatively be proved by the more \emph{ad hoc} methods used in~\cite{ksriley}.
\end{proof}

\begin{remark}
\rm {Proposition~\ref{prop:asympdirn} shows that $\P_{p/q} \neq \emptyset$ for all $p/q \in \QQ$. This can be viewed as a special case of the general result of~\cite{BonO} Theorem 1,
see also~\cite{chois} Theorem 2.4.  We have to be careful to include the case, excluded in \cite{BonO}, that the group $G_{\S}(x)$ is Fuchsian so that $\C/G$ has zero volume.
The conclusion is the following:
\begin{proposition} Let $\gamma$  be an essential simple non-peripheral closed curve on $\dd \S$. Then 
$\P_{\gamma} \neq \emptyset$ if and only if $\gamma$ is non-trivial in $\pi_1(\S)$ and intersects the meridian disk $\gamma_{1/0}$ at least twice.
If $\gamma$ meets  $\gamma_{1/0}$ exactly twice then the bending angle is identically $\pi$ and $G_{\S}(x)$ is Fuchsian.
\end{proposition}
}
\end{remark}

\subsection{ Step 6. Density of rational pleating rays}

Finally, we justify the claim that  the rational pleating rays are dense in $\D$:
\begin{theorem}[\cite{kstop} Corollary 6.2, \cite{ksriley} Theorem 5.2] \label{thm:density} Rational pleating rays  are dense in $\D_{\S}$.
\end{theorem}
\begin{proof} The proof of this result in any one complex dimensional parameter space   is the same. Here is a quick sketch.
Suppose that $\nu$ is an irrational lamination  with corresponding pleating variety $\P_{\nu}$, and that $x \in \P_{\nu} \cap \D$. Pick a sequence of rational measured laminations $\nu_n = c_n \delta_{\gamma_n}$ where $c_n \in \RR^+$ so that $\nu_n  \to \nu $ in the space of projective measured laminations on $\dd \S$, where $\delta_{\gamma_n}$ is the unit point mass on $\gamma_n$. Replace the traces of $\gamma_n$ by complex length functions $\lambda_n$ and  scale to get complex analytic functions $c_n \lambda_n$. One shows that
in a neighbourhood of $x \in \P_{\nu}$ these functions form a normal family  which converges to a non-constant analytic function (\cite {kstop} Theorem 20),   whose real locus 
contains the pleating ray $\P_{\nu}$ (\cite {kstop} Theorem 23). 
By   Hurwitz' theorem,  there are nearby points at which the approximating functions $c_n \lambda_n$ must take on real values. In a small enough neighbourhood of $x$, this is enough to force $ y \in \P_{\gamma_n}$ (\cite {kstop} Theorem 31). This gives density in $\Int \D$. By the result quoted in the introduction that $\D = \overline{\Int \D}$ we are done. \end{proof}

  \subsection{The pleating rays for $\H$} 
\label{sec:rays}

By Corollary~\ref{discretetogether}, $\D_{\H}= \D_{\S}$. Thus  the rational rays for $\D_{\S}$ are also dense in $\D_{\H}$. However it is easy to see that 
 a rational pleating laminations on  $\dd \H(x)$ correspond exactly to those on $\S(x)$,  and that although the actual bending curves differ, their traces are related by a simple formula.

\begin{lemma} \label{bendlinesagree} Suppose that the bending lamination  $\beta_{\H}(x)$ of $\H(x)$ is rational so that its support $\lambda$ is a union of disjoint simple closed curves on $\dd \H$.  
Let $\gamma$ be a connected component of 
$\lambda$. Then either $\O(\gamma) = \gamma$ or the three curves $\gamma, \O(\gamma) , \O^2(\gamma)$
are disjoint. The support of the bending lamination   $\beta_{\S}(x)$ is exactly the projection of $\gamma$ to $\S$ and all rational bending laminations of $\S $ arise in this way.
\end{lemma}
\begin{proof} The limit set of $G_{\H}(x)$ and hence its convex core are invariant under the symmetry  $\O$.
Hence  the support $\lambda$ of $\beta_{\H}(x)$   is also $\O$-invariant.
Let $\gamma$ be a connected component of 
$\lambda$. Since connected components of $\lambda$ are pairwise disjoint, either $\O(\gamma) = \gamma$ or the three curves $\gamma, \O(\gamma), \O^2(\gamma)$
are disjoint.
In either case,  $\g$ cannot pass through a fixed point of $\O$: at the fixed point $P$ the images of $\g$ would meet at angles $2 \pi/3$ so that $\gamma, \O(\gamma) , \O^2(\gamma)$ would intersect at $P$, which is impossible.

Let $\pi_{\O}$ be the projection $\H \to \S$. In a neighbourhood of a bending line $\pi_{\O}$ is a covering map hence a local isometry. Since being a bending line can be characterised locally, $\b_{\S}(x)$ is the projection of $\g$ to $\S$.

Let $\gamma$ be a simple closed curve on $\dd \S$.
Clearly, by the same observation about local characterisation of bending lines,  if $\gamma$ is a bending line then so is any connected component of its lift  to $\dd \H$. This proves the converse.  
 \end{proof}
 
 We remark that if $p/q$ is congruent to $1/0$ or $0/1$ mod $\ZZ_2$ then the lift of $\gamma_{p/q}$ has three connected components which are permuted among themselves by $\O$, while if $p/q$ is congruent to $1/1$ then its lift  has one $\O$-invariant connected component. 
 To see this, check by hand for the curves $\gamma_{1/0}, \gamma_{0/1}, \gamma_{1/1}$ and then note that the lifting property is invariant under the mapping class group of $\dd \S$ which at the same time acts transitively on $p/q$   congruence classes mod $\ZZ_2$.

To actually compute the pleating rays for $\D_{\S}$, we  computed the traces $\Tr V_{p/q}(x)$ corresponding to the  curves $\g_{p/q} \in \pi_1(\S)$. The above discussion shows that it is unnecessary to actually compute traces of lifted curves   in $\pi_1(\H)$. If for some reason one wanted to do this, either one could start again enumerating the curves on $\H$, or one could note that the complex length of a lift of $\g_{p/q}$ in $\H$ would be either the same as or three times that of the curve $\g_{p/q}$ in $  \S$, depending on the $\ZZ_2$-parity of $p/q$.

 \section{Computing traces} \label{sec:torustree}

  To compute traces of the elements $V_{p/q}$, rather than use the $\S$-tree as in Section~\ref{sec:traces},  we actually  performed computations on the  associated Markoff tree corresponding to the associated torus $\T$ of Section~\ref{sec:solidtorus}, referred to in this section as the $\T$-tree. To justify this, we need to compare the curves in Farey position $p/q$ on the two trees to ensure that they do indeed correspond geometrically as expected. We also need to address the issue about lifting representations  to $\SL$ raised in Remark~\ref{signchoices1}.
 
  \subsubsection{Correspondence of curves}\label{sec:correspondence}

Homotopy classes of essential simple non-peripheral  loops on $\dd \T$ are  well known to be in bijective correspondence to unoriented lines of rational slope  in the plane, see for example~\cite{serint,kstop}.
 In fact the word $W_{p/q}$ generated by the concatenation process following the $\T$-tree described in Section~\ref{sec:markoff} is the cutting sequence of a line of slope $p/q \in \hat \QQ$ across the lattice, see~\cite{serint}.  
 \begin{figure}[hbt] \begin{center}
 \includegraphics[height=6cm]{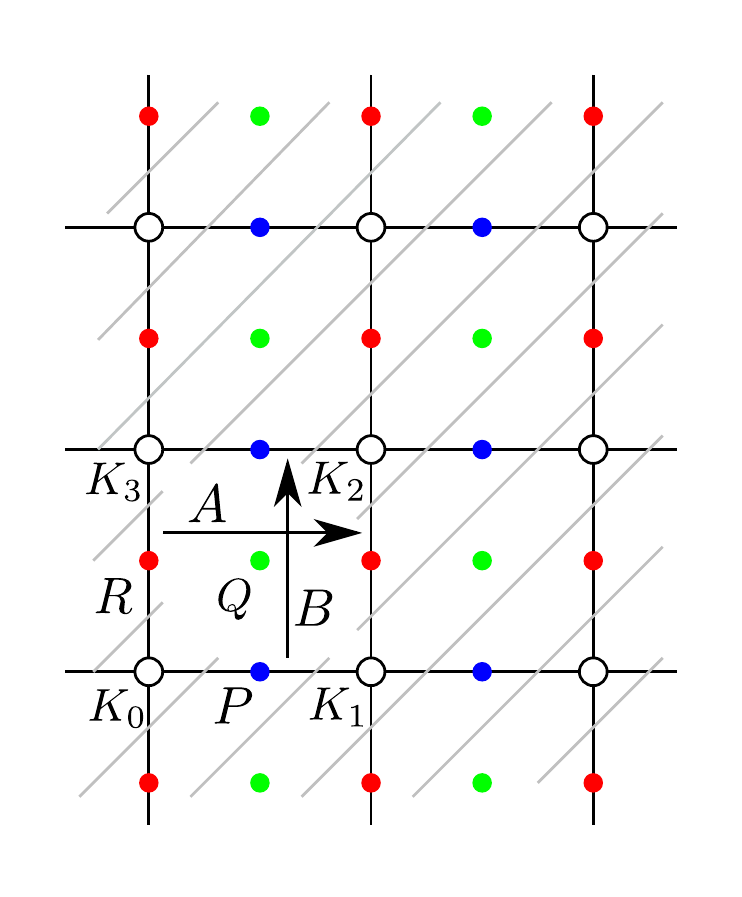}  
 \caption{Lattice representation of a cover of $\dd \S$. The integer  vertices (white circles) correspond to the end points of the order $3$ axes on $\dd \S$; the endpoints of the order $2$ elliptics $P,Q, R$ are coloured blue, green and red respectively.} \label{fig:lattice}
\end{center}\end{figure}

The key point here is that the plane with a cone singularity of angle $2\pi/3$ at integer lattice points, see Figure~\ref{fig:lattice}, is an intermediate covering between the universal cover $\HH$ of $\dd \T$ and $\dd \T$ itself.   As described in for example~\cite{ksriley}, the same lattice can also be viewed as an intermediate covering between 
$\HH$ and $\dd \S$: the rectangle with vertices at $0, 1, 2i, 2i+1$  projects bijectively to $\dd \S$ while the rectangle with vertices   $0, 1,   i/2, 1+  i/2$ projects bijectively to $\dd \U$ and the unit square projects  bijectively to   the torus $\dd \T$. The lattice points correspond to the cone points belonging to $K_i, i=1, \ldots, 4$ arranged as shown. Thus there  is also a bijective correspondence between lines of rational slope in the punctured plane and simple essential non-peripheral curves on  $\dd \S$. In this way, one can easily relate the words $W_{p/q}$ (on $\dd \T$) and $ V_{p/q}$ (on $\dd \S$); this is explained in detail in~\cite{ksriley}. 

In this picture,  the meridian loop $\dd m$ of Section~\ref{sec:enumeration} is identified as the `vertical' line of slope $1/0$.  
One sees easily that the line of slope $p/q$ in the plane projects to a curve on $\dd \S$ which has exactly $2q$ intersections with $\dd m$ and a twist of $p$ as described in Section~\ref{sec:enumeration}.   It follows from Lemma~\ref{lem:treeisok}  that  the labelling of curves by lines of rational slope $p/q$  exactly corresponds  to  the Farey labelling of curves  by their position on the $\S$-tree.

As above, the curve in Farey position $p/q$ on the $\S$-tree is denoted $\g_{p/q}$, corresponding to a word $V_{p/q}$; while the curve in Farey position $p/q$ on the $\T$-tree is denoted $\w_{p/q}$, corresponding to a word $W_{p/q}$. Now $  \S$ projects to $\U$ by a four-fold cover and  $\T$ projects to $\U$ by a two-fold cover. Hence we have:

\begin{proposition}\label{coverings}
The complex length of $\g_{p/q} $ is twice that of  $\w_{p/q}$, hence $   \Tr V_{p/q}(\z) = \pm (\Tr W_{p/q} (\z))^2 -2 )$. \end{proposition} 
Note that this allows for an ambiguity in the signs of the traces since the two lifts of $\pi_1(\T)$ and $\pi_1(\S)$ to $\SL$ are not (indeed cannot be) chosen consistently. 
\begin{corollary} Up to sign, the trace of  $\g_{p/q} \in \pi_1(\U)$   may be computed using the formula of Proposition~\ref{coverings} and  the $\T$-tree.  \end{corollary}

Since we are aiming to compute pleating rays which are a geometrical construct and hence only depend on a $\PSL$ representation, this would be sufficient for our purposes. However it is more  satisfying    to prove the following more precise result which  shows that working with the $\SL$ lift of the representation of $\pi_1(\T)$ described in Section~\ref{sec:solidtorus}, we can fix the choice of sign.

 \begin{prop} \label{relatingtraces} With $W_{p/q}, V_{p/q}$  as above, let $f_{p/q} (z) = \Tr V_{p/q}(\z)$  and  $g_{p/q}  (\z)  = \Tr W_{p/q} (\z) $.  Then $-f_{p/q} (\z) =(g_{p/q}  (\z))^2 -2$ for all $p/q \in \hat \QQ$.   \end{prop} 
 \begin{proof} It is easy to check that this is correct for $p/q = 0/1,1/0, 1/1$.
 In detail: 
 
 $\w_{0/1} = A, \gamma_{0/1}  = K_0K_1$ and we have shown that $  A^2 = -K_0K_1$.
Thus $ f_{0/1} (\z) =x$, $(g_{p/q}  (\z))^2 -2 = (-x+2)-2 = -x$.
 
 $\w_{1/0} =B, \gamma_{1/0}  = \rm{id}$ and $B^2 = - \rm{id}$.
So $ f_{1/0} (\z) =2$, $(g_{1/0}  (\z))^2 -2 = -2$.
 
  $\w_{1/1} =AB, \gamma_{1/1}  = \rm{K_0K_1^{-1}}$.
So $ f_{1/1} (\z) =1-x$, $(g_{1/1}  (\z))^2 -2 = x-1$.

Now we do an inductive proof. 
Suppose that in the $\S$-tree  labels $u,v$ are adjacent along an edge $e$ with $w$ the remaining label at one of the two vertices at the ends of $e$. By the formula in Section~\ref{sec:traces}
 the  label at the other vertex is $2 - uv -w$.

 Suppose that the corresponding labels on the $\T$-tree  are $u',v',w'$. Then the remaining label at the vertex at the other end of $e$ is  $ u'v'-w'$. Replace these labels by the negatives of the traces of the doubled curves to get labels
 $2-u'^2, 2-v'^2, 2-w'^2 ,2- (u'v'-w')^2 $ around the same $4$ vertices. If we can show that 
 $$2 -(2-u'^2)(2-v'^2) - (2-w'^2 )  = -   (2- (u'v'-w')^2 )    $$ we will be done.
This is easily checked by multiplying out, noting that the trace identity~\eqref{conedspheretraces} round a vertex of the $\T$-tree gives
 $$ u'^2 +   v'^2   + w'^2 = u'v'w' + \Tr [A,B] +2 =  u'v'w' + 3.$$  \end{proof}
%the left hand side is
 %$$2 -(2-u'^2)(2-v'^2) - (2-w'^2 ) = - u'^2 v'^2 + 2u'^2 +  2v'^2 -4 + w'^2$$
% while the right hand side is 
% $$-   2+  u'^2 v'^2 + w'^2 - 2u'v'w'$$
 % The trace identity round a vertex of the $\T$-tree gives
% $$ u'^2 +   v'^2   + w'^2 = u'v'w' + \Tr [A,B] +2 =  u'v'w' + 3.$$
%Thus  \begin{eqnarray*}
%&& - 2+  u'^2 v'^2 + w'^2 - 2u'v'w' = -   2+  u'^2 v'^2 + w'^2  - 2(u'^2 +   v'^2   + w'^2 -3) \\ &&= 
%4 +  u'^2 v'^2- 2(u'^2 +   v'^2) - w'^2\end{eqnarray*} proving the result.
% \end{proof}

  \subsubsection{The actual computations}\label{sec:torustree1}

The above discussion justifies the method we actually used to perform  computations involving traces on $\S$. 
Instead of computing on the $\S$-tree with initial traces   $\Tr \g_{0/1} =  \Tr  K_0K_1 = x, \Tr \g_{1/0}= \Tr   \mbox{\rm{id}}= 2, \Tr \g_{1/1} = \Tr K_0K_1^{-1}  = 1-x$, we used the
 $\T$-tree  with initial triple
$\Tr A = \pm i(3/\z - \z),  \Tr B =0$ and $\Tr AB = \pm i(3/\z + \z)$ corresponding to the generators $A,B$ of $G_{\T}$.  As in Section~\ref{sec:solidtorus}
 Observe that $A^2 = -K_0K_1$, so that $\Tr A^2 = -x$. Since $\Tr B = 0$, we can find $\Tr AB$ from the  identity 
$(\Tr A)^2 +  (\Tr AB)^2 = \Tr {[A,B]} +2 = 3$. 
Thus setting $(a,b,c) = (\Tr A, \Tr B, \Tr AB)$ we have
$$
a^2 - 2 = -x, \quad c^2  =   1 + x.
$$
It is easily checked that this is in accord with~\eqref{eqn:torustraces}.
Thus associated to $G_{\T}(x)$  we have the torus tree
 $(a,0,c) = (\sqrt{-x+2}, 0, \sqrt{x+1})$. This is the method we actually used to compute the 
 pleating rays shown in Figure~\ref{Diagonal-and-Riley-mu-3-Ray-BQ}.

\begin{remark} \label{rmk:signs} \rm{The sign of the square roots in the above can be  uniquely determined by the formulae for traces in terms of $\xi$. What we actually did was to make an arbitrary choice and plot rays corresponding to curves in the range $0 \leq p/q \leq 1$, thus making a picture in the upper half plane which we could then reflect. As can be seen from Figure~\ref{fig:trace2}, the  signs of the square roots in fact  alternate periodically with period $4$ rather than period $2$, so that, for example, $\Tr \g_{3} = -\Tr \g_{1}$.}
\end{remark}

\subsubsection{Computations for the Riley slice} \label{sec:discussion}  The traces needed to find the pleating  rays for the Riley slice on the right  in Figure~\ref{Figs/Riley-Ray-BQ} were computed by a method similar to that described above.   Our parameter $x$ can be related to the parameter $\rho$ of~\cite{ksriley} 
 by comparing the traces of the word in Farey position $0/1$: these are $K_0K_1$ in our case and $XY$ in the notation of~\cite{ksriley}.   Thus we find   the correct correspondence is $ x \leftrightarrow \rho + 2 $. 
For the Riley group  a similar computation to the one above with $\Tr [A,B] = -2$ gives immediately  
 $(\Tr A)^2 = -(\rho + 2)$ and  $(\Tr AB)^2 =  \rho + 2 $. Thus writing in terms of the $x$-coordinate we find the initial triple $(\sqrt {-x}, 0, \sqrt {x})$.

\subsubsection{Comparison of Bowditch sets}  It is interesting to compare the Bowditch sets associated to  the two  initial triples $(x,x,x)$ and $(\sqrt{-x+2}, 0, \sqrt{x+1})$.  In the latter case, one needs to modify the definition of the Bowditch set: since $\phi(U)=0$ for some $U \in \Omega$, there is a trace preserving $\ZZ$-action on the associated tree $\TT_{(\sqrt{-x+2}, 0, \sqrt{x+1})}$ corresponding to the action of a subgroup of ${\rm Aut}(F_2)$ generated by a parabolic, see  for example \cite{tan_gd} Theorem 1.9.   The Bowditch condition should actually be specified  on $\Omega \setminus \{U\} / \sim$, where $\sim$ is   the equivalence coming from this symmetry.

The results, plotted in the $\z$-plane, are shown in Figure~\ref{fig:BQ-sets-comparison}. On the right  the initial triple is  $(x,x,x)$  (with $x$ related to $\xi$ as in \eqref{eqn:paramreln}) corresponding to the handlebody group $G_{\H}(x)$.
On the left,  the initial triple is $(\sqrt{-x+2}, 0, \sqrt{x+1})$  corresponding to the torus group $G_{\T}(x)$. 
The two regions are clearly distinct: the grey region on the right contains that on the left.
Conjecturally, the left hand grey region is also the discreteness locus for the groups $G_{\T}(x)$, see Figure \ref{Figs/Riley-Ray-BQ} for the parametrization in terms of $x$.

Note the various symmetries as discussed in Section~\ref{sec:symmetries}, in particular note how  Figure~\ref{Diagonal-BQ} loses the left-right reflectional symmetry seen in Figure~\ref{fig:BQ-sets-comparison}.
The coloured region in Figure~\ref{Diagonal-BQ} is the same region as the right frame of Figure~\ref{fig:BQ-sets-comparison},  drawn in the $x$-plane.

\begin{figure}[ht]\label{xxx}
\includegraphics[width=4cm]{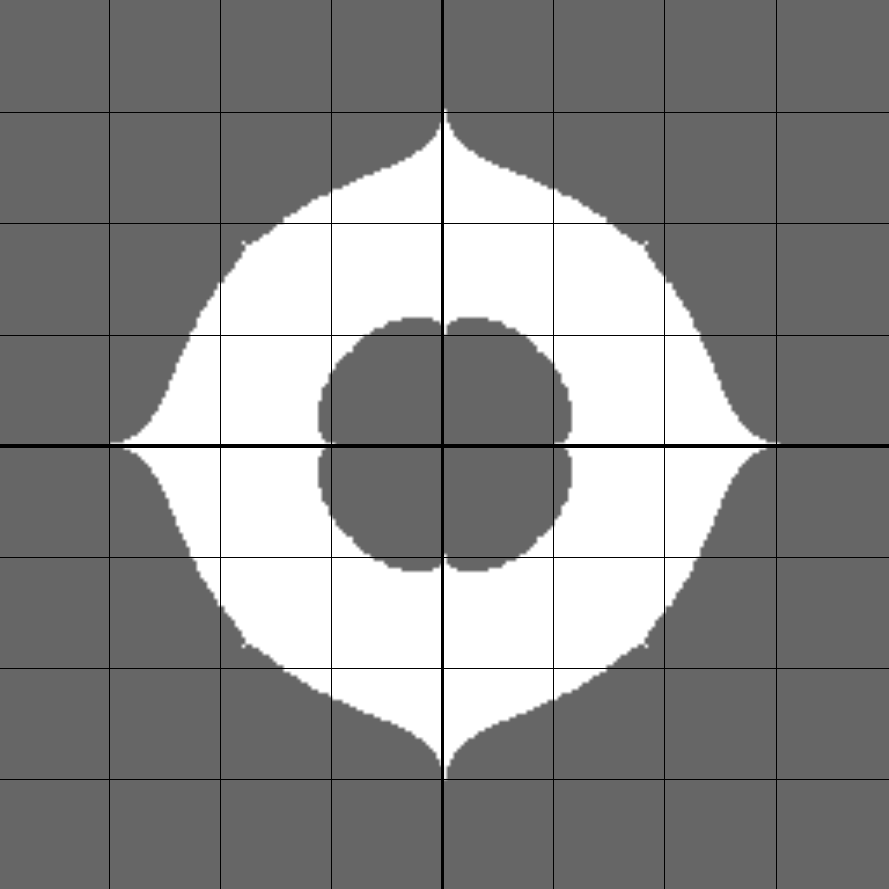}
\hspace{1cm}
\includegraphics[width=4cm]{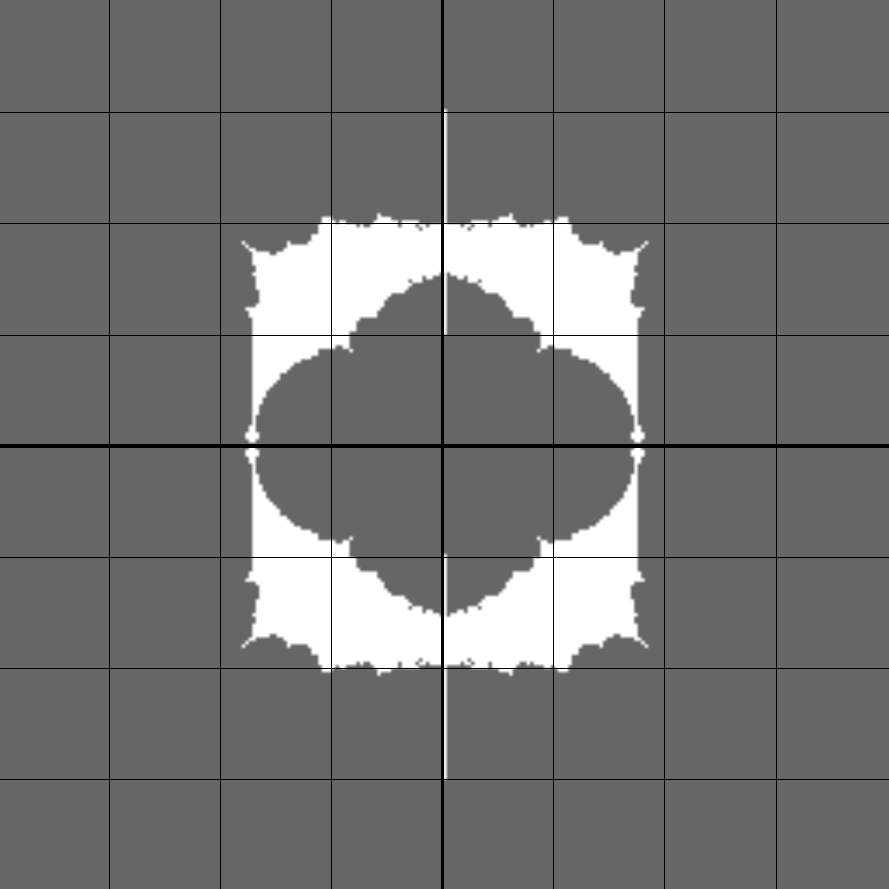}
\caption{Bowditch sets (grey) plotted in the $\z$-plane with range $[-4,4]\times[-4i,4i]$.  Left:  Initial triple  $(\sqrt{-x+2}, 0, \sqrt{x+1})$  
corresponding to the torus group $G_{\T}(x)$. Right: Initial triple  $(x,x,x)$ corresponding to the handlebody group $G_{\H}(x)$.
The two regions are clearly distinct: the grey region on the right contains that on the left. }  
\label{fig:BQ-sets-comparison}
\end{figure}

\bibliographystyle{alpha}

\begin{thebibliography}{000}%%%%%%%%%%%%%%%%%%%%%%%%%%%%%%%%%%%%%%%%%%%%%%% 


  \bibitem{akiyoshi}
 H. Akiyoshi, M Sakuma, M. Wada and  Y. Yamashita.
\newblock Punctured torus groups and $2$-bridge knot groups I.
\newblock {\em Springer Lecture Notes in Math. 1909}.  Springer, 2007.




\bibitem{BonO} F. Bonahon, J-P. Otal. 
\newblock Laminations
  mesur\'ees de plissage des vari\'et\'es hyperboliques de dimension
  3. 
  \newblock {\em Ann. of Math. 160}, 1013--1055, 2004.
  
  
  
%  \bibitem{bow_apr}
%B.~H. Bowditch.
%\newblock A proof of {M}c{S}hane's identity via {M}arkoff triples.
%\newblock {\em Bull. London Maths. Soc.  28}, 73--78,
%  1996.

%\bibitem{bow_ava}
%B.~H. Bowditch.
%\newblock A variation of {M}c{S}hane's identity for once-punctured torus
 % bundles.
%\newblock {\em Topology 36}, 325--334,  1997.

\bibitem{bow_mar}
B.~H. Bowditch.
\newblock {M}arkoff triples and quasi-{F}uchsian groups.
\newblock {\em Proc. London Math. Soc.  77}, 697--736, 1998.


 \bibitem{canary}
R. Canary.
\newblock Pushing the boundary.
\newblock In  {\em In the Tradition of Ahlfors and Bers, III. Contemporary Math. Vol 355}, W. Abikoff, A.~Haas eds.,
AMS Publications,  109--121, 2004.



   \bibitem{chois}   Y. Choi and C. Series. 
            \newblock  Lengths are coordinates for convex structures.
\newblock {\em J. Diff. Geom. 73}, 75 -- 116, 2006.


   \bibitem{Culler} M. Culler.
         \newblock  Lifting representations to covering groups.
\newblock {\em Advances in Math. 59},    64 -- 70, 1986.
 



\bibitem{EpM} D. B. A. Epstein and A.  Marden. \newblock Convex hulls in hyperbolic
  space, a theorem of Sullivan, and measured pleated surfaces.  
  In {\em
    Analytical and geometric aspects of hyperbolic space.}
  London Math. Soc. Lecture Note
  Ser., 111, Cambridge Univ. Press, Cambridge, 113--253, 1987.   
 
  \bibitem{Fenchel}
W. Fenchel.
\newblock {\em Elementary geometry in hyperbolic space}, Vol.~11 of {\em de
  Gruyter Studies in Mathematics}.
\newblock Walter de Gruyter \& Co., Berlin, 1989.
 



\bibitem{goldman}
W. Goldman. 
\newblock The modular group action on real $SL(2)$-characters of a one-holed torus.
\newblock {\em Geometry and Topology}  7, 443 -- 486, 2003.


\bibitem{goldman2}
W. Goldman. 
\newblock Trace coordinates on Fricke spaces of some
simple hyperbolic surfaces.
\newblock In {\em Handbook of Teichm\"uller theory Vol. II},  IRMA Lect. Math. Theor. Phys., 13, Euro. Math. Soc., Z\"urich, 611-- 684, 2009.

\bibitem{GMST}
W. Goldman, G. McShane, G. Stantchev and S.P. Tan. 
\newblock Dynamics of the automorphism group of the
two-generator free group on the space of
isometric actions on the hyperbolic plane.
\newblock {\em In preparation}, 2014

  
  \bibitem{hempel}
J. Hempel.
\newblock {\em $3$-manifolds.}
\newblock {\em   Ann. of Math. Studies 86}.
\newblock Princeton Univ. Press, 1976.




\bibitem{HK}
C. Hodgson and S. Kerckhoff.
\newblock Rigidity of hyperbolic cone-manifolds and hyperbolic Dehn
surgery.
\newblock {\em J. Differential Geometry  48}, 1--59, 1998.



\bibitem{kstop}
L.~Keen and C.~Series.
\newblock Pleating coordinates for the {M}askit embedding of the
   {T}eichm{\"u}ller space of punctured tori.
\newblock {\em Topology  32}, 719 --749, 1993.

\bibitem{ksriley}   L. Keen and C. Series. 
\newblock The Riley slice of Schottky space. 
\newblock {\em Proc. London Math. Soc. 69},  72 -- 90, 1994.


\bibitem{kscont}
  L. Keen and C. Series. 
\newblock Continuity of convex hull boundaries.
 \newblock {\em Pacific J. Math. 168},  183 -- 206, 1995.

\bibitem{kshowtobend}
 L. Keen and C. Series. 
\newblock How to bend pairs of punctured tori. 
\newblock In {\em Lipa's Legacy},  J.~Dodziuk and L.~Keen eds,
Contemporary Math. 211,   359 -- 387, 1997.



            
\bibitem{ksqf}
L.~Keen and C.~Series. 
\newblock Pleating invariants for punctured torus groups.
\newblock {\em Topology  43},    447 -- 491, 2004.


\bibitem{koms} Y. Komori and C. Series.  
\newblock The Riley slice revisited. 
In {\em The Epstein Birthday Schrift},  I.~Rivin, C.~Rourke and C.~Series eds.,
Geom. and Top. Monographs, 
            Vol.1, International Press, 303 -- 316,  1999. 
            
            
\bibitem{kra}
I.~Kra. 
\newblock On lifting Kleinian groups to $\SL$.
\newblock In  {\em Differential Geometry and Complex Analysis}, I.~Chavel, H.~Farkas eds.,
 Springer,  181 -- 193, 1985.



   
    
 
 
  
  \bibitem{MPT}
S.~Maloni,  F. Palesi and S.P. Tan. 
\newblock On the character variety of the four-holed sphere.
\newblock {\em Groups, Geometry and Dynamics},   to appear (2014).
          
            
           
\bibitem{NT}
S.P.K. Ng and S.P. Tan.
\newblock{The complement of the Bowditch space in the ${\rm SL}(2, {\mathbb C})$ character variety.}
\newblock{\em Osaka J. Math.} 44, 247--254, 2007.
            
 \bibitem{otal} J-P. Otal.  \newblock Sur le coeur convexe d'une
  vari\'et\'e hyperbolique de dimension 3.  Unpublished preprint,
  1994.

  
  
  
\bibitem{serint}      C. Series. 
 \newblock The geometry of Markoff numbers.
\newblock In \emph{Math. Intelligencer   7},  20 -- 29, 1985.
 


 \bibitem{ser-wolp}
   C. Series. \newblock  An extension of Wolpert's derivative formula.
\newblock {\em Pacific J. Math.  197},  223 -- 239,  2001.

 

%\bibitem{seriesbmc} C. Series.  
%\newblock Thurston's bending measure conjecture  for once punctured torus
%groups. 
%\newblock In \emph{Spaces of Kleinian Groups},  Y.~Minsky, M.~Sakuma and C.~Series eds., LMS Lecture Notes 329, Cambridge Univ. Press   75 -- 90, 2006.


%\bibitem{sty} C. Series, S.P. Tan, Y. Yamashita. 
% \newblock Experiments with the Bowditch set. 
% \newblock In preparation.
 
 
% \bibitem{tan_gen2}
%S.P.Tan, Y. L.  Wong  and Y. Zhang.
%\newblock Generalizations of {M}c{S}hane's identity to hyperbolic
%  cone-surfaces.
%\newblock {\em Journal of Differential Geometry  72}, 73--112, 2006.

\bibitem{tan_gd}
S.P.Tan, Y. L.  Wong  and Y. Zhang.
\newblock Necessary and sufficient conditions for McShane's identity and variations.
\newblock {\em Geometriae Dedicata, 119 }, 199--217, 2006.


\bibitem{tan_gen}
S.P.Tan, Y. L.  Wong  and Y. Zhang.
\newblock Generalized {M}arkoff maps and {M}c{S}hane's identity.
\newblock {\em Adv. Math.  217}, 761--813, 2008.



\bibitem{wada}
M. Wada.
\newblock OPTi's algorithm for discreteness determination.
\newblock {\em Experimental Math}, 15:1--124, 2006.

\bibitem{OPTi}
\url{http://delta.math.sci.osaka-u.ac.jp/OPTi/}

%%%%%%%%%%%%%%%%%%%%%%%%%%%%%%%%%%%%%%%
%%%%%%%%%%%%%%%%%%%%%%%%%%%%%%%%%%%%%%%

\end{thebibliography}

%%%%%%%%%%%%%%%%%%%%%%%%%%%%%%%%%%%

\end{document}